\documentclass[a4paper,11pt]{article}
\pdfoutput=1

\usepackage[utf8]{inputenc}

\usepackage{theorem}
\usepackage{amsfonts}
\usepackage{amsmath}
\usepackage{amssymb}
\usepackage{bbm}

\usepackage{graphicx}
\usepackage{tikz}
\usetikzlibrary{arrows}
\usepackage{subfigure}
\usepackage[margin=25mm]{geometry}

\usepackage{enumitem}
\usepackage[authoryear]{natbib}
\usepackage{ifthen}

\newcommand{\eps}{\varepsilon}
\newcommand{\R}{\mathbb{R}}
\newcommand{\PP}{\mathbb{P}}
\newcommand{\dd}[1][]{\mathrm{d}#1}
\newcommand{\obs}{\mathrm{obs}}
\newcommand{\intv}{\mathrm{int}}
\newcommand{\nobs}{n_{\obs}}
\newcommand{\nint}{n_{\intv}}
\newcommand{\mId}{\mathbb{I}}
\newcommand{\transp}{\mathrm{T}}
\newcommand{\smallbullet}{\tikz[baseline]\draw[fill=black] (0, 0.4ex) circle (0.15ex);}

\DeclareMathOperator{\Cov}{Cov}
\DeclareMathOperator{\pa}{pa}

\DeclareMathOperator{\diag}{diag}
\DeclareMathOperator{\doop}{do}
\DeclareMathOperator{\tr}{tr}

\DeclareMathOperator{\SHD}{SHD}
\DeclareMathOperator*{\argmin}{arg\,min}

\DeclareMathAlphabet{\mathsc}{OT1}{cmr}{m}{sc}

\newtheorem{theorem}{Theorem}
\newtheorem{lemma}[theorem]{Lemma} 
\newtheorem{proposition}[theorem]{Proposition} 

\newtheorem{corollary}[theorem]{Corollary}
\newtheorem{definition}{Definition}

\newenvironment{subprop}{\begin{enumerate}[label=(\roman*), ref=(\roman*), nolistsep]}{\end{enumerate}}

\newcommand{\SiMy}{\ensuremath{\mathsc{SiMy}}}
\newcommand{\GIES}{\ensuremath{\mathsc{GIES}}} 

\newcommand{\interior}[1]{\stackrel{\circ}{#1}}

\newcommand{\spforall}{\; \forall \;}

\newcommand{\spst}{\;|\;}

\newcommand{\mvn}{\mathcal{N}}

\newcommand{\mcI}{\ensuremath{\mathcal{I}}}

\newcommand{\data}{\mathcal{T}, \mathbf{X}}

\newcommand{\sR}{\mathbb{R}}

\newcommand{\incr}[1]{\textrm{d}#1}
\newcommand{\supscr}[1]{\ensuremath{^{\mathrm{#1}}}}
\newcommand{\subscr}[1]{\ensuremath{_{\mathrm{#1}}}}

\newcommand{\partder}[2]{\frac{\partial #1}{\partial #2}}

\DeclareMathAlphabet{\mathsc}{OT1}{cmr}{m}{sc}

\newlength{\edgelength}
\DeclareRobustCommand{\grarright}{\mathbin{\tikz[baseline] \draw[->] (0pt, 0.7ex) -- (\edgelength, 0.7ex);}}
\DeclareRobustCommand{\grarleft}{\mathbin{\tikz[baseline] \draw[<-] (0pt, 0.7ex) -- (\edgelength, 0.7ex);}}

\newlength{\exgredge}
\newenvironment{exgraphpicture}{%
  \begin{tikzpicture}[baseline=(v1.base)]
    \node[anchor=mid] (v1) at (0, 0) {$1$};
    \node[anchor=mid] (v2) at (\exgredge,  0) {$2$};
    \node[anchor=mid] (v3) at (2\exgredge, 0) {$3$};
    \node[anchor=mid] (v4) at (3\exgredge, 0) {$4$};
    \node[anchor=mid] (v5) at (\exgredge,  -\exgredge) {$5$};
    \node[anchor=mid] (v6) at (2\exgredge, -\exgredge) {$6$};
    \node[anchor=mid] (v7) at (3\exgredge, -\exgredge) {$7$};
}{\end{tikzpicture}}

\newenvironment{proof}[1][]{\par\noindent{\bf Proof\ifthenelse{\equal{#1}{}}{}{ of #1}\ }}{\hfill$\square$\\[2mm]}

\title{Jointly interventional and observational data: estimation of interventional Markov equivalence classes of directed acyclic graphs} 
\author{Alain Hauser and Peter Bühlmann\\
{\small\texttt{\{hauser, buhlmann\}@stat.math.ethz.ch}} \\
       Seminar für Statistik\\
       ETH Zürich\\
       8092 Zürich, Switzerland}
\date{}

\begin{document}

\maketitle

\begin{abstract}
  In many applications we have both observational and (randomized) interventional data. We propose a Gaussian likelihood framework for joint modeling of such different data-types, based on global parameters consisting of a directed acyclic graph (DAG) and correponding edge weights and error variances. Thanks to the global nature of the parameters, maximum likelihood estimation is reasonable with only one or few data points per intervention. We prove consistency of the BIC criterion for estimating the interventional Markov equivalence class of DAGs which is smaller than the observational analogue due to increased partial identifiability from interventional data. Such an improvement in identifiability has immediate implications for tighter bounds for inferring causal effects. Besides methodology and theoretical derivations, we present empirical results from real and simulated data. 
\end{abstract}

Keywords: Causal inference; Interventions; BIC; Graphical model; Maximum likelihood estimation; Greedy equivalence search
 
\section{Introduction}
\label{sec:introduction}

Causal inference often relies on an underlying influence diagram in terms of a directed acyclic graph (DAG). In absence of knowledge of the true underlying DAG, there has been a substantial line of research to estimate the Markov equivalence class of DAGs which is identifiable from data. Most often, the target of interest is the observational Markov equivalence class to be inferred from observational data; that is, the data arises from observing a system in ``steady state'' without any interventions, see for example the books by \citet{Spirtes2000Causation} or \citet{Pearl2000Causality}. For the important case of  multivariate Gaussian distributions, the observational Markov equivalence class is rather large and thus, many parts of the true underlying DAG are unidentifiable from observational data, see for example \citet{Verma1990Equivalence} or \citet{Andersson1997Characterization} for a graphical characterization of the Markov equivalence class in the Gaussian or the fully nonparametric case. Under additional assumptions, identifiability of the whole DAG is guaranteed as with linear structural equation models with non-Gaussian errors \citep{Shimizu2006Linear} or additive noise models \citep{Hoyer2009Nonlinear}, see also \citet{Peters2011Identifiability}. 

In many applications, we have both observational and interventional data, where the latter are coming from (randomized) intervention experiments. In biology, for example, we often have observational data from a wildtype individual and interventional data from mutants or individuals with knocked-out genes.  Besides the methodological issue of properly modeling such data, we gain in terms of identifiability: the interventional Markov equivalence class is smaller \citep{Hauser2012Characterization}, thanks to additional interventional experiments, and this is of particular interest for the Gaussian and nonparametric cases which are hardest in terms of identifiability. 

We focus here on the problem of joint modeling of observational and interventional Gaussian data. Thereby, we assume that the observational distribution is Markovian \citep[and typically faithful; cf.][]{Spirtes2000Causation} to a true underlying DAG $D_0$ and that the different interventional distributions are linked to the DAG $D_0$ and the observational distribution via the intervention calculus using the do-operator \citep{Pearl2000Causality}. Linking all interventional distributions to the same DAG $D_0$ and the single observational distribution allows to deal with the situation where we have only one interventional data point for every intervention target (intervention experiment). We propose to use the maximum likelihood estimator which has not been studied or even used for the observational-interventional data setting. We prove that when penalizing with the BIC score, it consistently identifies the true underlying observational-interventional Markov equivalence class. 

\subsubsection{Relation to other work}

Some approaches to incorporate interventional data for learning causal models have been developed in earlier work.  \citet{Cooper1999Causal} and \citet{Eaton2007Exact} address the problem of calculating a posterior (and also a likelihood) of a data set having observational as well as interventional data but do not investigate properties of the Bayesian estimators e.g.\ in the large-sample limit nor address the issue of identifiability or Markov equivalence.  \citet{He2008Active} present a method which first estimates the observational Markov equivalence class and then in a second step, it identifies additional structure using interventional data. This technique is inefficient due to decoupling into two stages, especially if one has many interventional but only a few observational data: in fact, our maximum likelihood estimator in Section \ref{sec:structure-learning} can cope with the situation where we have interventional data only.  To our knowledge, no analysis of the maximum likelihood estimator of an ensemble of observational and interventional data has been pursued so far.  The computation of the maximum likelihood estimator which we will briefly indicate in Section \ref{sec:simulations} has been developed in \citet{Hauser2012Characterization}: due to its non-trivial nature, it is not dealt with in this paper.  When having observational data only, the works by \citet{Chickering2002Learning,Chickering2002Optimal} are dealing with maximum likelihood estimation and consistency of the BIC score for the corresponding observational Markov equivalence class: however, the extension to the mixed interventional-observational case, which occurs in many real problems, is a highly non-trivial step.

\section{Interventional-observational data and maximum likelihood estimation}

We start by presenting the model and the corresponding maximum likelihood estimator.

\subsection{A Gaussian model}

We consider the setting with $\nobs$ observational and $\nint$
interventional $p$-variate data from the following model:
\begin{align}
    & X^{(1)},\ldots ,X^{(\nobs)}\ \mbox{i.i.d.}\ \sim P_{\obs},\nonumber\\
    & X_{\intv}^{(1)},\ldots , X_{\mathrm{int}}^{(n_{\intv})}\ \mbox{independent,
    and independent of $X^{(1)},\ldots ,X^{(\nobs)}$},\ X_{\intv}^{(i)} \sim
    P_{\intv}^{(i)}. \label{eqn:basic-model}
\end{align}
In the following, we specify the observational distribution
$P_{\obs}$ and all the interventional distributions
$P^{(i)}_{\intv}\ (i=1,\ldots ,\nint)$. 

Regarding the observational distribution, we assume that 
\begin{eqnarray}
    \label{eqn:pobs}
    P_{\obs} = {\cal N}_p(0,\Sigma),\ \mbox{where}\ P_{\obs}\ \mbox{is
    Markovian with respect to a DAG}\ D. 
\end{eqnarray}
The assumption with mean zero is not really a restriction: all derivations
can be easily adapted, at the price of writing an intercept in
many formulas. An implementation in the \texttt{R}-package \texttt{pcalg}
\citep{Kalisch2010Causal} offers the option to restrict to mean zero or not. The
Markovian assumption is equivalent to the factorization property in
(\ref{eqn:factorization}) below. 
We sometimes refer to the true observational distribution as $P_{0,\obs}$
with parameter $\Sigma_0$, and the true DAG is $D_0$.  

In the following, the set of nodes in a DAG $D$, associated to the
$p$-dimensional random vector $(X_1, \ldots, X_p)$, is denoted 
by $\{1,\ldots ,p\}$ and the parental set by
$$
  \pa(j) = \pa_D(j) = \{k;\ k\ \text{a parent of node}\ j\}\ (j=1,\ldots ,p) \ .
$$
The Markov condition of
$P_{\obs}$ with respect to the DAG $D$, with parental sets $\pa(\cdot) =
\pa_D(\cdot)$, allows the following (minimal) factorization of the joint
distribution \citep{Lauritzen1996Graphical}: 
\begin{eqnarray}
    \label{eqn:factorization}
    f_{\obs}(x) = \prod_{j=1}^p f_{\obs}(x_j|x_{\pa(j)}),
\end{eqnarray}
where $f_\obs(\cdot)$ denotes the Gaussian density of $P_\obs$
and $f_\obs(x_j|x_{\pa(j)})$ are univariate Gaussian conditional
densities.  

The interventional distributions $P^{(i)}_\intv\ (i=1,\ldots ,\nint)$ may all
be different but linked to the same observational distribution $P_{\obs}$ and
the same DAG $D$ via the intervention calculus in Section
\ref{sec:docalc}. Due to the common underlying model given by $P_{\obs}$
and the DAG $D$, this 
allows to handle cases where we have only one interventional data point for
every interventional distribution. 

\subsubsection{Intervention calculus}
\label{sec:docalc}

The intervention calculus, or do-calculus \citep{Pearl2000Causality}, is a key concept for describing the model of the intervention distributions. We consider the DAG $D$ appearing in the observational model (\ref{eqn:pobs}), and we assign it a \emph{causal} interpretation as follows.  Assume $X_\intv$ is realized under a (single- or multi-variable) intervention at the intervention target $I \subseteq \{1, \ldots, p\}$ denoting the set of intervened vertices.  The distribution of $X_\intv$ is then given by the so-called truncated factorization, a version of the factorization in (\ref{eqn:factorization}). The truncated factorization for the interventional distribution for $X_\intv$ with deterministic intervention $\doop(X_I = u_I)$ is defined
as \citep{Pearl2000Causality}:
$$
    f_{\intv}(x_{I^c}|\doop(X_I = u_I)) = 
    \prod_{j \notin I} f_{\obs}(x_j|x_{\pa(j) \cap I^c}, u_{\pa(j) \cap I}),
$$
where $f_{\intv}(\cdot|\doop(X_I = u_I)$ is the intervention Gaussian
density when doing an intervention at $X_I$ by setting it to the value
$u_I$, and $f_{\obs}(\cdot|\cdot)$ is as in (\ref{eqn:factorization}). 
Here, the conditioning argument $x_{\pa(j) \cap I^c}, u_{\pa(j) \cap I}$
distinguishes the value of the unintervened variables $x_{\pa(j) \cap I^c}$
and the values of the intervened variables $u_{\pa(j) \cap I}$.

Deterministic interventions as described above make the intervened variables $X_I$ deterministic, having the value of the intervention levels $u_I$.  In this paper, we consider stochastic interventions where the intervened variables $X_I$ are set to the value of a random vector $U_I \sim \prod_{j \in I} f_{U_j}(u_j) \dd{u_j}$ with independent (but not necessarily identically distributed) components having densities $f_{U_j}(\cdot)\ (j \in I)$. The truncated 
factorization for stochastic interventions (where the intervention values
are independent of the observational variables) then reads as follows:
\begin{equation}
    \label{eqn:trunc-factorization}
    f_\intv(x|\doop(X_I = U_I)) = \prod_{j \notin I} f_{\obs}(x_j|x_{\pa(j) \cap
    I^c}, U_{\pa(j) \cap I}) \prod_{j \in I} f_{U_j}(x_j).
\end{equation}
In contrast to the case of deterministic interventions above, the intervention density (\ref{eqn:trunc-factorization}) is $p$-variate: $x \in \R^p$ and for $j \in I$, $x_j$ is then an argument in the density from the random intervention variable $U_j$. In the following, we assume that the densities for the intervention values are Gaussian as well: 
\begin{equation}
    \label{eqn:Udist}
U_1,\ldots,U_p\ \mbox{independent with}\ U_j \sim {\cal
  N}(\mu_{U_j},\tau_j^2)\ (j=1,\ldots ,p).
\end{equation} 

The truncated factorization in (\ref{eqn:trunc-factorization}) or its
deterministic version above can be obtained by applying the Markov property
to the interventional DAG $D_I$: given a DAG $D$, the intervention DAG
$D_I$ is defined as $D$ but deleting all directed edges which point into
$i \in I$, for all $i \in I$. 

An interventional data point $X_\intv^{(i)}$, with intervention target
$T^{(i)} = I \subseteq \{1,\ldots ,p\}$ and corresponding intervention
value $U_{I}^{(i)}$, then has density $f_{\intv}(x|\doop(X_{I}^{(i)}  =
U_I^{(i)}))$ from (\ref{eqn:trunc-factorization}). Thus, in other words, the
intervention distribution $P_{\intv}^{(i)}$ is characterized by the
Gaussian density in (\ref{eqn:trunc-factorization}). This, together with the
specific form of the Gaussian observational distribution (see also
(\ref{eqn:factorization})), fully specifies the model in (\ref{eqn:basic-model})
which then reads as:
\begin{align}
    &X_{\obs}^{(1)},\ldots ,X_{\obs}^{(\nobs)}\ \mbox{i.i.d.}\ \sim f_{\obs}(x)\dd{x}\ \mbox{as in (\ref{eqn:factorization})},\nonumber\\
    &X_{\intv}^{(1)},\ldots , X_{\mathrm{int}}^{(n_{\intv})}\ \mbox{independent,
    and independent of $X_\obs^{(1)}, \ldots, X_\obs^{(\nobs)}$}, \nonumber\\
    & X_{\intv}^{(i)} \sim f_{\intv}(x|\doop(X_{T^{(i)}} = U_{T^{(i)}}^{(i)})) \dd{x}\
    \mbox{as in (\ref{eqn:trunc-factorization})}\nonumber\\ 
    & U^{(1)}, \ldots, U^{(n_{\intv})}\ \mbox{independent,
    and independent of $X_\obs^{(1)}, \ldots, X_\obs^{(\nobs)}$}, \nonumber\\
    & U^{(i)} \sim \mvn\left(\mu^{(i)}_U, \diag(\tau_1^{(i)2}, \ldots, \tau_p^{(i)2})\right)
    \label{eqn:basic-model2}
\end{align}
The true underlying parameters and quantities in the model
(\ref{eqn:basic-model2}) are denoted by $\mu_0$, $\Sigma_0$, $\mu^{(i)}_{0,U}$, $\{\tau_{0,j}^{(i)2}\}_j$
and the true DAG $D_0$. It is well known, see also Section
\ref{sec:structure-learning}, that $D_0$ is typically not identifiable from the
observational and a few interventional distributions. 

\subsubsection{Structural equation model}

The model in (\ref{eqn:basic-model2}) (or in (\ref{eqn:basic-model})) can be alternatively written
as a linear structural equation model thanks to the Gaussian
assumption. The observational variables can be represented as
\begin{equation}
    \label{eqn:structural-eq1}
    X_{\obs,k} = \sum_{j=1}^p \beta_{kj} X_{\obs,j} + \eps_k,\ \eps_k \sim
    \mvn(0, \sigma^2_k)\ (k = 1, \ldots, p),
\end{equation}
where $\beta_{kj} = 0$ if $j \notin \pa(k) = \pa_D(k)$ and $\eps_1,\ldots
,\eps_n$ are independent and $\eps_k$ independent of
$X_{\obs,\pa(k)}$. Using the matrix $B = (\beta_{kj})_{k,j=1}^p$ with 
\begin{equation}
    \label{eqn:weight-matrices}
    B \in \mathbf{B}(D) := \{A = (\alpha_{kj}) \in \R^{p \times p};\
    \alpha_{kj} = 0\ \mbox{if}\ j \notin \pa_D(k)\},
\end{equation}
we can write 
$$
    X_{\obs} = B X_{\obs} + \eps,\ \eps \sim \mvn_p(0,\mbox{diag}(\sigma_1^2, \ldots,\sigma_p^2)).
$$

An interventional setting with intervention $\doop(X_I = U_I)$ (and
intervention target $T = I$) can be
represented as follows: 
\begin{equation}
    \label{eqn:structural-eq2}
    X_{\intv,k} = 
    \begin{cases}
        \sum_{j \notin I} \beta_{kj} X_{\intv,j}
        + \sum_{j \in I} \beta_{kj} U_j + \eps_k &, \text{if }
        k \notin I, \\
        U_k &, \text{if } k \in I,\\
    \end{cases}
\end{equation}
with $\beta_{kj}$ and $\eps_k$ as in (\ref{eqn:structural-eq1}) with the additional property that $U$ is independent of $X_{\obs}$ and $\eps$.

Thus, the model in (\ref{eqn:basic-model2}) is given as 
\begin{align}
    &X_{\obs}^{(1)},\ldots ,X_{\obs}^{(\nobs)}\ \mbox{i.i.d. as in
    (\ref{eqn:structural-eq1})},\nonumber\\ 
    &X_{\intv}^{(1)},\ldots , X_{\mathrm{int}}^{(n_{\intv})}\ \mbox{independent,
    and independent of } X^{(1)}, \ldots, X^{(\nobs)}, \nonumber\\
    &X_{\intv}^{(i)}\ \mbox{as in (\ref{eqn:structural-eq2}) with intervention target}\
    I = T^{(i)},\nonumber\\
    & U^{(1)}, \ldots, U^{(n_{\intv})}\ \mbox{independent,
    and independent of $X_\obs^{(1)}, \ldots, X_\obs^{(\nobs)}$}, \nonumber\\
    & U^{(i)} \sim \mvn\left(\mu^{(i)}_U, \diag(\tau_1^{(i)2}, \ldots, \tau_p^{(i)2})\right)
    \label{eqn:structural-eq3}
\end{align}
It also holds that for $\eps$ in (\ref{eqn:structural-eq1}), $\eps^{(1)},\ldots
,\eps^{(n)}$ are independent of $U^{(1)},\ldots ,U^{(n)}$. As before, we
denote the true underlying quantities by $B_0$, $\{\sigma_{0,k}^2\}_k$, $\mu_{0,U}$, $\tau_0^2$ and the true DAG $D_0$.  Because of the causal interpretation of the DAG model, we call a model as in (\ref{eqn:structural-eq3}) or (\ref{eqn:basic-model2}) a Gaussian causal model in the following.

\subsection{Maximum likelihood estimation when the DAG is given}
\label{sec:mle-given-dag}

The likelihood for the Gaussian model (\ref{eqn:basic-model2}) is parameterized
by the covariance matrix $\Sigma$ of $P_{\obs} = {\cal N}_p(0,\Sigma)$,
the DAG $D$ and the parameters $\mu^{(i)}_U,\tau^{(i)2}$ for the stochastic 
intervention values $U^{(i)}$. Alternatively, and the route taken here, we can
use the linear structural equation model and parameterize the likelihood
with the coefficient matrix $B$, the error variances $\sigma_1^2,\ldots
,\sigma_p^2$, and $\mu^{(i)}_U,\tau^{(i)2}$. Using this, the matrix $B$ is constrained
such that its non-zero elements are corresponding to the directed edges in
the DAG $D$.

For a given DAG $D$, it is rather straightforward to derive the maximum
likelihood estimator, as discussed below. Much more involved is the issue
of structure learning when the DAG $D$ is unknown: there we want to
estimate a suitable Markov equivalence of the unknown DAG, as discussed in
Section \ref{sec:structure-learning}.

It is easy to see that the log-likelihood for $\mu^{(i)}_{U_j}$, $\tau_j^{(i)2}$ decouples from the remaining parameters, and we
regard $\mu^{(i)}_{U_j},\ \tau_j^{(i)2}$ (for all $i$ and $j$) as nuisance parameters. 

In the sequel, we unify the notation and denote an observational data point
with the intervention target $I = \emptyset$. We can then write 
the distribution of $X_\intv|\doop(X_I = U_I)$ as:
\begin{align}
    X | \doop(X_I & = U_I) \sim {\cal N}(\mu^{(I)}, \Sigma^{(I)}), \label{eqn:int-dist} \\
    \mu^{(I)} & = \textstyle \left(\mId - R^{(I)} B\right)^{-1} Q^{(I)\transp} \mu_{U_I}, \nonumber \\
    \Sigma^{(I)} & = \textstyle \left(\mId - R^{(I)} B\right)^{-1} \left[R^{(I)} \diag(\sigma^2) R^{(I)} + Q^{(I)\transp} \diag(\tau_I^2) Q^{(I)}\right] \left(\mId - R^{(I)} B\right)^{-\transp}. \nonumber
\end{align}
Thereby, we have used the following matrices:
\begin{alignat}{3}
    & P^{(I)}: & \ \R^p &\to \R^{p - |I|}, & x & \mapsto x_{I^c}, \nonumber \\
    & Q^{(I)}: & \R^p &\to \R^{|I|}, & x & \mapsto x_I, \label{eqn:help-matrices} \\
    & R^{(I)}: & \R^p &\to \R^p, & R^{(I)} & := P^{(I)\transp} P^{(I)}. \nonumber
\end{alignat}
The Gaussian distribution in (\ref{eqn:int-dist}) is a direct consequence of (\ref{eqn:structural-eq2}), which can be rewritten in vector-matrix notation as
$$
    X_\intv = R^{(I)} \left(BX_\intv + \eps\right) + Q^{(I)\transp} U \ .
$$
Denoting the intervention target for the $i$th data point $X^{(i)}$ by
$T^{(i)}$, and the total
sample size as $n = n_{\obs} + n_{\intv}$, the log-likelihood (conditional
on $U^{(1)},\ldots ,U^{(n)}$) becomes
$$
  \ell_D(B, \{\sigma_k^2\}_k, \{\mu^{(i)}_U\}_i, \{\tau^{(i)2}\}_i; T^{(1)}, \ldots, T^{(n)}, X^{(1)}, \ldots , X^{(n)}) 
  = \sum_{i=1}^n \log f_{\mu^{(T^{(i)})},\Sigma^{(T^{(i)})}}(X^{(i)}),
$$
where $f_{\mu^{(T^{(i)})},\Sigma^{(T^{(i)})}}$ denotes the density of $\mvn(\mu^{(T^{(i)})}, \Sigma^{(T^{(i)})})$ in (\ref{eqn:int-dist}) which depends on $B$, $\{\sigma_k^2\}_k$, $\{\mu^{(i)}_U\}_i$ and $\{\tau^{(i)2}\}_i$.  To make notation shorter, we will denote by $\mathcal{T}$ the sequence of intervention targets $T^{(1)}, \ldots, T^{(n)}$ in the following, and by $\mathbf{X}$ the data matrix consisting of the rows $X^{(1)}$ to $X^{(n)}$.

For a given DAG structure $D$, implying certain zeroes in $B \in
\mathbf{B}(D)$ through the space $\mathbf{B}(D)$ in (\ref{eqn:weight-matrices}), the
maximum likelihood estimator is defined as: 
\begin{equation}
  \label{eqn:MLED}
  \hat{B}(D),\{\hat{\sigma}_k^2(D)\}_k = \argmin_{\substack{B \in \mathbf{B}(D)\\\{\sigma_i^2\} \in (\R^{+})^p}} - \ell_D(B,\{\sigma_i^2\}_i; \mathcal{T}, \mathbf{X}). 
\end{equation}
The expressions $\hat{B}(D),\{\hat{\sigma}_k(D)^2\}_k$ have an explicit
form as described in Section \ref{sec:proofs-mle}; the nuisance parameters $\{\mu^{(i)}_U\}_i, \{\tau^{(i)2}\}_i$ do not appear in (\ref{eqn:MLED}) any more since the minimizer of the likelihood does not depend on them.

\section{Estimation of the interventional Markov equivalence class}
\label{sec:structure-learning}

Consider the model in (\ref{eqn:basic-model2}) or (\ref{eqn:structural-eq3}). It is well known that one cannot identify the underlying DAG $D_0$ from $P_{\obs} = P_{0,\obs}$. However, assuming e.g. faithfulness of the distribution as in (\ref{eqn:pobs}), one can identify the observational Markov equivalence class ${\cal E}(D_0) = {\cal E}(P_{\obs})$ from $P_{\obs}$, see for example \citet{Spirtes2000Causation} or \citet{Pearl2000Causality}. 

\subsection{Characterizing the interventional Markov equivalence class}
\label{sec:markov-equivalence-class}

The power of interventional data is that we can identify more than the
observational Markov equivalence class, namely the smaller
interventional Markov equivalence classes \citep{Hauser2012Characterization}. Regarding the
latter, we consider a family of intervention targets, a subset of the 
powerset of the vertices $\{1, \ldots, p\}$: $\mcI \subset \mathcal{P}(\{1, \ldots, p\})$. 
In our context $\mcI = \{T^{(i)} \subseteq \{1, \ldots, p\};\ i=1, \ldots, n\}$ is the set of intervention targets of the $n_{\intv}$ interventional data together with the empty set $\emptyset$ as long as we have at least one observational data point ($\nobs > 0$).

A family of targets \mcI{} is called conservative if for all $j \in \{1, \ldots, p\}$, there is some $I \in {\cal I}$ such that $j \notin I$. The simplest such family is $\mcI = \{\emptyset\}$, i.e., observational data only. Furthermore, every \mcI{} arising from an ensemble of observational and interventional data is a conservative family of targets as well. The issue that a family of targets should be conservative is crucial for characterization of interventional Markov equivalence classes \citep{Hauser2012Characterization}, and with having jointly observational and interventional data in mind, it is not really a restriction. 

The example in Figure \ref{fig:ex-equivalent-dags} shows three DAGs that are observationally Markov equivalent since they have the same skeleton (i.e., they yield the same undirected graph if all directed edges are replaced by undirected ones) and the same v-structures (i.e., induced subgrahps of the form $a \grarright b \grarleft c$) \citep{Verma1990Equivalence}.  If we have, in addition to observational data, data from an intervention at vertex $4$, the orientations of the arrows incident to the intervened vertex become identifiable.  Technically speaking, the interventional Markov equivalence class under the family of targets $\mcI = \{\emptyset, \{4\}\}$ is \emph{smaller} than the observational Markov equivalence class.

\begin{figure}[t]
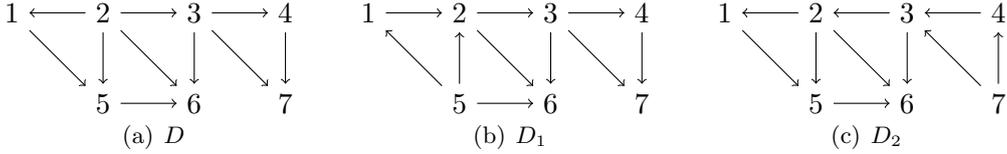

  \centering
  \subfigure[$D$]{%
    \begin{exgraphpicture}
      \draw[->] (v2) -- (v1);
      \draw[->] (v2) -- (v3);
      \draw[->] (v3) -- (v4);
      \draw[->] (v1) -- (v5);
      \draw[->] (v2) -- (v5);
      \draw[->] (v2) -- (v6);
      \draw[->] (v3) -- (v6);
      \draw[->] (v5) -- (v6);
      \draw[->] (v3) -- (v7);
      \draw[->] (v4) -- (v7);
    \end{exgraphpicture}
  } \quad
  \subfigure[$D_1$]{%
    \begin{exgraphpicture}
      \draw[->] (v5) -- (v1);
      \draw[->] (v1) -- (v2);
      \draw[->] (v5) -- (v2);
      \draw[->] (v2) -- (v3);
      \draw[->] (v3) -- (v4);
      \draw[->] (v2) -- (v6);
      \draw[->] (v3) -- (v6);
      \draw[->] (v5) -- (v6);
      \draw[->] (v3) -- (v7);
      \draw[->] (v4) -- (v7);
    \end{exgraphpicture}
  } \quad
  \subfigure[$D_2$]{%
    \begin{exgraphpicture}
      \draw[->] (v2) -- (v1);
      \draw[->] (v3) -- (v2);
      \draw[->] (v4) -- (v3);
      \draw[->] (v7) -- (v3);
      \draw[->] (v7) -- (v4);
      \draw[->] (v1) -- (v5);
      \draw[->] (v2) -- (v5);
      \draw[->] (v2) -- (v6);
      \draw[->] (v3) -- (v6);
      \draw[->] (v5) -- (v6);
    \end{exgraphpicture}
  }
  \caption{Three DAGs having equal skeletons and a single v-structure, $3 \grarright 6 \grarleft 5$, hence being observationally Markov equivalent.  Under the family of intervention targets $\mcI = \{\emptyset, \{4\}\}$, $D$ and $D_1$ are still (interventionally) Markov equivalent (i.e., statistically indistinguishable), while $D$ and $D_2$ belong to different interventional Markov equivalence classes.}
  \label{fig:ex-equivalent-dags}
\end{figure}

The general definition of an interventional Markov equivalence class is given in Section \ref{sec:def-int-equivalence-class}. The interventional Markov equivalence class ${\cal E}_{\cal I}(D_0)$ is identifiable from $P_{0,\obs}$ in (\ref{eqn:pobs}) and the interventional distributions, given by $f_{\intv}(x|\doop(X_{I} = U)) dx$ in (\ref{eqn:basic-model2}) for all $I \in {\cal I}$, assuming faithfulness as in assumptions (A1) and (A2) below.  In \citet{Hauser2012Characterization}, the interventional Markov equivalence class of a DAG $D$ for a conservative family of intervention targets \mcI{} is rigorously characterized in terms of a chain graph with directed and undirected edges, the so-called interventional essential graph or \mcI-essential graph.

\subsection{Structure learning using BIC-score}
\label{sec:structure-learning-bic}

For estimating the structure and the parameters of the interventional
Markov-equivalence class, we consider the penalized maximum-likelihood
estimator using the BIC-score. Denote by $\hat{B}(D)$ and
$\{\hat{\sigma}_k^2(D)\}_k$ the maximum-likelihood estimators for a given DAG
$D$, as in (\ref{eqn:MLED}). An estimate for the interventional
Markov-equivalence class is then:
\begin{align}
  & \hat{{\cal E}}_\mcI = \argmin_{\mathcal{E_I}(D)} - \ell_D(\hat{B}(D),
  \{\hat{\sigma}_k^2(D)\}_k; \data) + \frac{1}{2} \log(n) \dim({\cal E}_{\cal
  I}(D)),\label{eqn:MLEEQ}\\
  & \dim({\cal E}_{\cal I}(D)) = \dim(D) = \mbox{number of
  non-zero elements in}\ \hat{B}(D).\nonumber
\end{align}
The optimization is over all interventional Markov equivalence classes
with corresponding DAGs $D$, see also Section \ref{sec:computation} below. 

We note that the $\ell_0$-penalty has the property that the score remains
invariant for all members in the interventional Markov equivalence class
${\cal E}_{\cal I}(D)$: this property is not true for some other penalties such
as the $\ell_1$-norm. We outline in Section \ref{sec:computation} a
computational algorithm for computing the estimator in (\ref{eqn:MLEEQ}). 

We now justify the estimator in (\ref{eqn:MLEEQ}) by providing a consistency
result. We make the following assumptions.
\begin{description}
  \item[(A1)] The true bservational distribution $P_{0,\obs}$ in (\ref{eqn:pobs}), or equivalently the distribution of $X_{\obs} \sim f_\obs(x)\dd{x}$ in (\ref{eqn:basic-model2}) is faithful with respect to the true underlying DAG $D_0$.
  
  \item[(A2)] The true interventional distributions of $X_\intv^{(i)} \sim f_\intv(x|\doop(X_{T^{(i)}} = U^{(i)})) \dd{x}$ in (\ref{eqn:basic-model2}) are faithful with respect to the true underlying intervention DAG $D_{0,T^{(i)}}$, for all $i=1, \ldots, \nint$ (for the definition of the intervention DAG, see Section \ref{sec:docalc}.
\end{description}
The faithfulness assumption means that all marginal and conditional
independencies can be read off from the DAG, here $D_0$ or $D_{0,T^{(i)}}$,
respectively \citep{Spirtes2000Causation}. This is a stronger requirement than a Markov
assumption which allows to infer some conditional independencies from the
DAG $D_0$ or $D_{0,T^{(i)}}$.

In our case with a data set arising from different interventions, we do not have identically distributed data, as it is evident for example from equation (\ref{eqn:int-dist}).  To be able to make a precise consistency statement for the estimator (\ref{eqn:MLEEQ}), we regard the sequence of intervention targets as \emph{random}:
\begin{description}
  \item[(A3)] The intervention targets $T^{(1)}, \ldots, T^{(n)}$ are $n$ i.i.d.\ realizations of a random variable $T$ taking values in \mcI: $P[T = I] = w_I > 0$ for all $I \in \mcI$.
\end{description}
In Section \ref{sec:mle-given-dag}, we have already seen that the parameters $\mu^{(i)}_{U_j},\ \tau_j^{(i)2}$ (for all $i$ and $j$) are nuisance parameters.  They do not belong the statistical model, but describe the experimental setting (i.e., the interventions).  With assumption (A3), we introduce an additional, ``artificial'' set of nuisance parameters describing the experimental setting.  By this approach, we can model the sequence $(T^{(i)}, X^{(i)})_{i = 1}^n$ as independent realizations of random variables $(T, X) \in \mcI \times \sR^p$ with the following distribution:
$$
    P[T = I] = w_I, \quad f(x \spst T = I) = f_\intv(x \spst \doop(X_I = U_I))\ .
$$
\begin{theorem}
  \label{thm:consistency}
  Consider model (\ref{eqn:basic-model2}) with the family of intervention targets \mcI{}.  Assume (A1), (A2) and (A3).  Then:  as $n \to \infty$,
  $$
    \PP[\hat{\cal E}_\mcI = {\cal E}_\mcI(D_0)] \to 1.
  $$
\end{theorem}
A proof is given in Section \ref{sec:proof-consistency}. The result might not be surprising in view of model selection consistency results of BIC for curved exponential family models \citep{Haughton1988Choice}. However, a careful analysis is needed to cope with the special situation of data arising from different interventions and hence different distributions.

\bigskip\noindent
\textbf{Remark 1.} A version of Theorem \ref{thm:consistency} also holds without the
faithfulness assumptions (A1) and (A2). 

We define an \emph{independence map} as a DAG $D^*$ such that the observational distribution in (\ref{eqn:pobs}) (or equivalently the distribution of $X_{\obs} \sim f_{\obs}(x)\dd{x}$ in (\ref{eqn:basic-model2})) and all interventional distributions of $X_{\intv}^{(i)} \sim f_{\intv}(x|\doop(X_{T^{(i)}} = U^{(i)})) \dd{x}$ in (\ref{eqn:basic-model2}), for all $T^{(i)}$, can be generated by $D^*$ and the corresponding intervention DAGs $D^*_{T^{(i)}}$.  This is a generalization of an independence map for observational data \citep{Pearl1988Probabilistic}.  A \emph{minimum} independence map is an independence map having a minimum number of edges.  A minimum independence map is typically not unique, and assuming faithfulness in (A1) and (A2), the set of all minimum independence maps equals the interventional Markov equivalence class ${\cal E}_{{\cal I}_{\emptyset}}(D_0)$ with ${\cal I} = \{T^{(i)};\ i=1,\ldots ,\nint\}$. 

Instead of (\ref{eqn:MLEEQ}) consider the estimator 
$$
  \hat{D} = \argmin_D - \ell_D(\hat{B}(D), \{\hat{\sigma}_k^2(D)\}_k; \data) + \frac{1}{2} \log(n) \dim(D),
$$
where the optimization is over all DAGs $D$. The statement in Theorem
\ref{thm:consistency} can then be replaced by:
$$
\PP[\hat{D}\ \mbox{is a minimum independence map}]
\to 1.
$$

\bigskip\noindent
\textbf{Remark 2.} Although we have data sets with both observational and interventional data in mind, note that Theorem \ref{thm:bic-consistency} only makes the assumption of a \emph{conservative} family of intervention targets.    In other words, consistent model selection is even possible with interventional data alone.

\bigskip
Let $I \in \mcI \setminus \{\emptyset\}$ be an intervention target, and denote by $n_I = |\{i; T^{(i)} = I, i = 1, \ldots, n\}|$ the number of interventional data for this target.  Assumption (A3) made in the theorem implies $n_I \asymp n \to \infty$.  This might not be realistic in practice since there is often only one (or very few) interventional data point for each target $I$, i.e., $n_I = 1$ (or $n_I$ is small). Without having a rigorous proof, the consistency result of Theorem \ref{thm:consistency} is expected to hold if the intervention value $U^{(i)}$ is far away from zero, i.e., far away from the mean of $X_{T^{(i)}}$. The heuristics can be exemplified as follows.

\medskip\noindent
\textbf{Example 1.} Consider a DAG $D_0 = 1 \grarright 2$ with $p=2$ and
corresponding observational distribution from the structural equation model
\begin{align*}
  &X_1 \sim {\cal N}(0,\sigma_1^2),\\
  &X_2 \leftarrow \beta X_1 + \eps_2,\ \eps_2 \sim {\cal N}(0,\sigma_2^2).
\end{align*}
Then, the interventional distribution with target $I=1$ equals:
\begin{eqnarray}
\label{e1}
X_2|\doop(X_1 = u) \sim {\cal N}(\beta u,\sigma_2^2),
\end{eqnarray}
whereas the marginal observational distribution is
\begin{eqnarray}\label{e2}
X_2 \sim {\cal N}(0,\sigma_1^2 + \beta^2 \sigma_2^2).
\end{eqnarray}
Thus, if $u \to \infty$, the means of the distributions in (\ref{e1}) and
(\ref{e2}) drift away from each other and one realization from the
intervention in (\ref{e1}) would be sufficient such that with probability
tending to 1 as $u \to \infty$, we could detect the difference
from (one or many realizations of) the observational distribution  in
(\ref{e2}). 

Alternatively, if $u = 0$, we could detect differences of the distributions
in (\ref{e1}) and (\ref{e2}) in terms of their variances. But we would need
many realizations from (\ref{e1}) and (\ref{e2}) to detect this difference
with high probability. 

Although obvious, we note that if the true DAG would be $1 \grarleft 2$,
the distribution in (\ref{e1}) and (\ref{e2}) would coincide (being equal
to ${\cal N}(0,\mathrm{Var}(X_2))$. Therefore, when doing an intervention
$\doop(X_1=u)$ and we see a difference in comparison to the marginal
distribution of $X_2$, the true DAG must be $1 \grarright 2$. 

\medskip
We refer to empirical results in Section \ref{sec:simulations} which confirm
good model selection properties if $\nobs$ is large, $n_I = 1$ but with
intervention values $U$ chosen sufficiently far away from zero. 

\subsection{Computation}
\label{sec:computation}

The computation of the estimator in (\ref{eqn:MLEEQ}) is a highly non-trivial
task. The main difficulty comes from the fact we have to optimize over all
Markov equivalence classes. We can reformulate the optimization as follows:
$$
  \hat{B}, \{\hat{\sigma}_k^2\}_k = \argmin_{B \in {\bf B}_{\mathrm{DAG}};
  \{\sigma_k^2\}_k} - \ell(B,\{\sigma_k^2\}_k; \data) + \frac{1}{2} \log(n) \dim(B)
$$
where $-\ell(\cdot; \data)$ is the negative log-likelihood in the model
(\ref{eqn:structural-eq3}), and ${\bf B}_{\mathrm{DAG}}$ is the space of
matrices satisfying the constraint that they correspond to a DAG.
This DAG-constraint causes the optimization to be highly non-convex. In 
view of this, the $\ell_0$-penalty is not adding major new computational
challenges (and in fact allows for dynamic programming optimization, see
below) while it enjoys nice statistical properties and leading to a
score (value of the objective function) which is the same for all DAG
members in an interventional Markov equivalence class.

Somewhat surprisingly, although the optimization problem in (\ref{eqn:MLEEQ})
is NP-hard \citep{Chickering1996Learning}, dynamic programming can be 
used for exhaustive optimization \citep{Silander2006Simple}, roughly
as long as $p$ is less than say 20. For problems with larger dimension,
the optimization in (\ref{eqn:MLEEQ}) can be pursued using
greedy algorithms. Based on the idea from \citet{Chickering2002Learning,Chickering2002Optimal}, one can
use a greedy forward, backward and turning arrows algorithm which pursues each
greedy step in the space of interventional Markov equivalence classes which
is the much more appropriate space than the space of DAGs. An efficient
implementation of such an algorithm, called Greedy Interventional
Equivalent Search (GIES), is rigorously described in
\citet{Hauser2012Characterization} where algorithmic properties, 
theoretical and empirical, are
reported in detail. Although there is no guarantee that GIES converges to a
global optimum, it seems very competitive and keeps up with dynamic
programming for small-scale problems. An implementation of GIES is
available in the \texttt{R}-package \texttt{pcalg} which is used throughout
in Section \ref{sec:empirical-results}. 

\section{Empirical results}
\label{sec:empirical-results}

We evaluated $\ell_0$-penalized maximum likelihood estimation of interventional Markov equivalence classes as described in Section \ref{sec:structure-learning} on a real data set (Section \ref{sec:sachs}) as well as on simulated data (Section \ref{sec:simulations}).

\subsection{Analysis of protein-signaling data}
\label{sec:sachs}

We analyzed the protein-signaling data set of \citet{Sachs2005Causal}.  This data set contains 7466 measurements of the abundance of 11 phosphoproteins and phospholipids recorded under different experimental conditions in primary human immune system cells.  The different experimental conditions are characterized by associated reagens that inhibit or activate signaling nodes, corresponding to interventions at different points of the protein-signaling network.  Interventions mostly take place at more than one point, and the data set is purely interventional.  However, some of the experimental perturbations affect receptor enzymes instead of (measured) signaling molecules.  Since our statistical framework cannot cope with interventions at latent variables, we only considered 5846 out of the 7466 measurements which had an \emph{identical} perturbation of the receptor enzymes.  In this way, we model the system with perturbed receptor enzymes as ``ground state'', defining its distribution of molecule abundances as observational.

While we can make the data set fit our interventional framework by the aforementioned reduction to 5846 data points, the linear-Gaussian assumption of our framework may not hold, even after a log-transformation of the measurements.  Nevertheless, we fitted graphical models to the data set with different frequentist methods: GIES for the $\ell_0$-penalized MLE in (\ref{eqn:MLEEQ}) (see also Sections \ref{sec:computation} and \ref{sec:simulation-settings}), the PC algorithm \citep{Spirtes2000Causation}, the graphical LASSO \citep[GLASSO, ][]{Friedman2007Sparse}, and GIES combined with stability selection \citep{Meinshausen2010Stability}.  We varied the tuning parameter of each algorithm: the number of steps (i.e., of edge additions, deletions or reversals) in GIES, the significance level $\alpha$ in the PC algorithm, the penalty parameter $\lambda$ in GLASSO, and the cut-off selection probability in stability selection applied for GIES.  We compared the estimated models to the conventionally accepted model which serves as ground truth \citep{Sachs2005Causal}; the resulting ROC plots, both with respect to edge directions (defining true and false positives in terms of the graphs' adjacency matrices) and with respect to the skeleton alone, are depicted in Figure \ref{fig:sachs-roc}.

\begin{figure}
  \begin{center}
    \includegraphics{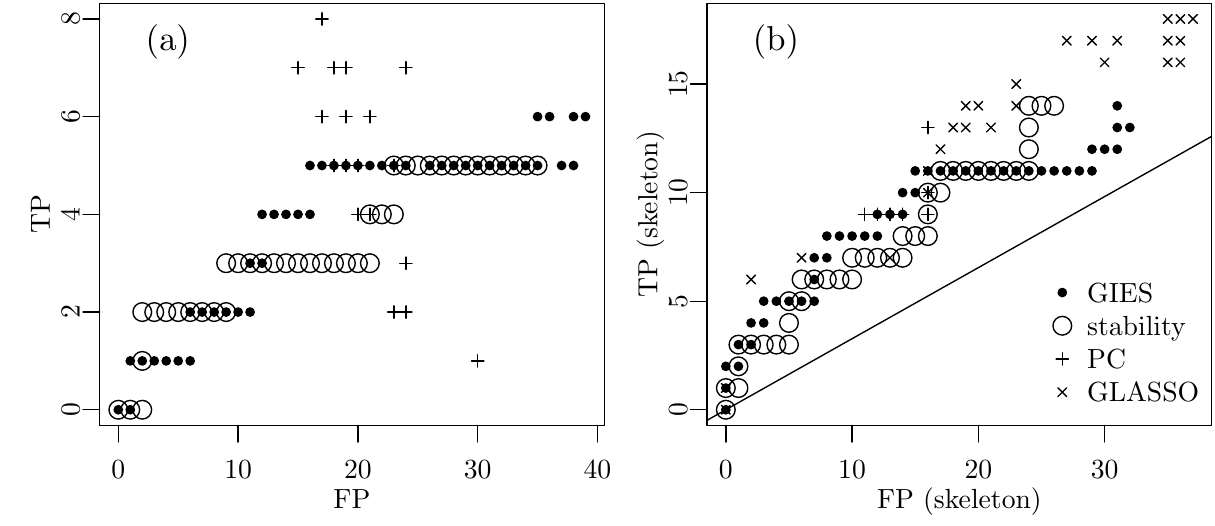}
  \end{center}
  \label{fig:sachs-roc}
  \caption{ROC plots of the models estimated from the Sachs data set, for directed edges (a) and the skeleton (b).  In (a), GLASSO is missing since it does not yield a directed model; in (b), random guessing is shown by a solid line.}
\end{figure}

The overall performance of the estimation of the \emph{skeleton} is comparable for all four algorithms (Figure \ref{fig:sachs-roc}(b)), even if two of them (PC and GLASSO) treat all data as identically distributed and disregard its interventional nature.  Regarding edge directions (Figure \ref{fig:sachs-roc}(a)), however, GIES (with or without stability selection) yields an improvement over the PC algorithm.

The Bayesian method of \citet{Cooper1999Causal} used for model fitting by \citet{Sachs2005Causal} is not directly comparable to the frequentist methods used here.  In particular, the results from \citet{Sachs2005Causal} are not easily reproducible due to choosing discretization levels and prior distribution.  Their performance as measured by comparison to the ground truth is substantially better than all methods considered in this paper ($15$ true positives, $7$ false positives in the convention of Figure \ref{fig:sachs-roc}(a)).  Potential reasons are increased robustness due to discretization and specific tuning (which is legitimate in their context of extending and improving the conventional ground truth).

\subsection{Simulations}
\label{sec:simulations}

We performed $\ell_0$-penalized maximum likelihood estimation as in (\ref{eqn:MLEEQ}) on interventional and observational data simulated from $4000$ randomly drawn Gaussian causal models (see (\ref{eqn:basic-model2}) or (\ref{eqn:structural-eq3})) to illustrate the consistency result of Theorem \ref{thm:consistency}.

\subsubsection{Experimental Settings}
\label{sec:simulation-settings}

We randomly drew DAGs whose skeleton has an expected vertex degree of $1.8$, $1.9$, $2.9$ and $3.9$ for $p = 10$, $20$, $30$ and $40$, respectively.  For every DAG $D$, we randomly generated a weight matrix $B \in \mathbf{B}(D)$ and error variances $\sigma_1^2, \ldots, \sigma_p^2$ such that the corresponding observational covariance matrix
$$
    \Sigma = \Cov(X_\obs) = (\mId - B)^{-1} \diag(\sigma^2) (\mId - B)^{-\transp}
$$
had a diagonal of $(1, \ldots, 1)$, meaning that each variable of the system had an observational marginal variance of $1$.  The procedure for generating Gaussian causal models of this form has been described in detail by \citet{Hauser2012Characterization}.

We simulated data sets with a total sample size $n = \nobs + \nint$ between $50$ and $10'000$.  We performed single-vertex interventions at $k$ randomly drawn vertices ($k = 0.2p$, $k = 0.5p$ and $k = p$), drawing $p/k$ samples under each intervention (that is, $5$, $2$ or $1$ for the chosen values of $k$).  These settings ensure that we had only $\nint = p$ interventional data points in each simulation, and that the majority of the data points were observational ones thus ($\nobs = n - p$).  This allowed us to verify our conjecture following Theorem \ref{thm:consistency} that few interventional samples are sufficient for consistent estimation of interventional Markov equivalence classes (or, equivalently, interventional essential graphs) as long as the intervention levels, the expectation values of the intervention variables $U$, are large enough.  In our simulations, we chose expecation values $\mu_{U_j}$ between $1$ and $50$ and variances of $\tau^2 = (0.2)^2$ for the intervention variables.  Note that because of the chosen normalization $\Sigma_{ii} = 1$, the expectation values $\mu_{U_j}$ can be thought of as being indicated in units of observational standard deviations.

To sum up, for each of the $4000$ randomly generated Gaussian causal models, we simulated $144$ data sets with observational and interventional data, namely one data set for each combination of the following experimental parameters:
\begin{itemize}
    \item $n \in \{50, 100, 200, 500, 1000, 2000, 5000, 10000\}$; $\nint = p, \nobs = n - p$;
    \item $k \in \{0.2p, 0.5p, p\}$;
    \item $\mu_{U_j} \in \{1, 2, 5, 10, 20, 50\}$.
\end{itemize}

We learned the structure of the underlying causal model from the simulated data sets using the BIC score as described in Section \ref{sec:computation}.  We used the two causal inference algorithms mentioned in Section \ref{sec:computation}:
\begin{itemize}
    \item an adaptation of the dynamic programming approach of \citet{Silander2006Simple} to interventional data which will be abbreviated as \SiMy{} in the following.  This algorithm guarantees to find the global minimizer of the BIC in (\ref{eqn:MLEEQ}); because of its exponential complexity, it is only applicable for models with no more than $20$ variables though.
    
    \item the Greedy Interventional Equivalence Search (\GIES) of \citet{Hauser2012Characterization}.  This algorithm \emph{greedily} optimizes the BIC score by traversing the search space of interventional Markov equivalence classes through operations corresponding to edge additions, deletions, or reversals in the space of DAGs.  The algorithm does not \emph{guarantee} to find the optimum of the BIC, but it was empirically shown for graphs with up to $p = 20$ nodes to have a performance comparable to that of \SiMy{} \citep{Hauser2012Characterization} while having polynomial runtime in the average case.
\end{itemize}

We assessed the quality of the estimated causal models with the structural Hamming distance SHD (\citealp{Tsamardinos2006Maxmin}; we use the slightly adapted version of \citealp{Kalisch2007Estimating}).  This quantity is a metric on the space of graphs.  The SHD between two graphs $G$ and $\hat{G}$ is the sum of false positives and false negatives of the skeleton and wrongly oriented edges.  Formally, if the graphs $G$ and $\hat{G}$ have adjacency matrices $A$ and $\hat{A}$, respectively, the SHD between $G$ and $\hat{G}$ is defined as
$$
    \SHD(G, \hat{G}) := \sum_{1 \leq i < j \leq p} \big(1 - \mathbbm{1}_{\{(A_{ij} = \hat{A}_{ij}) \wedge (A_{ji} = \hat{A}_{ji})\}}\big) \ .
$$

\subsubsection{Results}

Figure \ref{fig:shdess-n} shows the SHD between estimated and true interventional essential graph as a function of the total sample size $n$.  Results for different numbers of intervention targets showed similar characteristics (not shown).  The plots illustrate the consistency of the BIC, the main result of Theorem \ref{thm:consistency}.  As expected, convergence to the true equivalence class is faster the larger the intervention values (controlled by $\mu_U$) are.  Note, however, that the simulation setting does not fully match the limit setting of the theorem: while the theoretical result asks for the sample sizes $n_I$ of \emph{all} interventions $I \in \mcI$ to grow in the order $O(n)$, we always have $p$ interventional data points in our case while only the number of observational data points is growing.  In the setting with $n_I \asymp n$, \citet{Hauser2012Characterization} have already empirically shown the performance of \GIES{} as well as \SiMy{}.

\begin{figure}
  \begin{center}
    \includegraphics{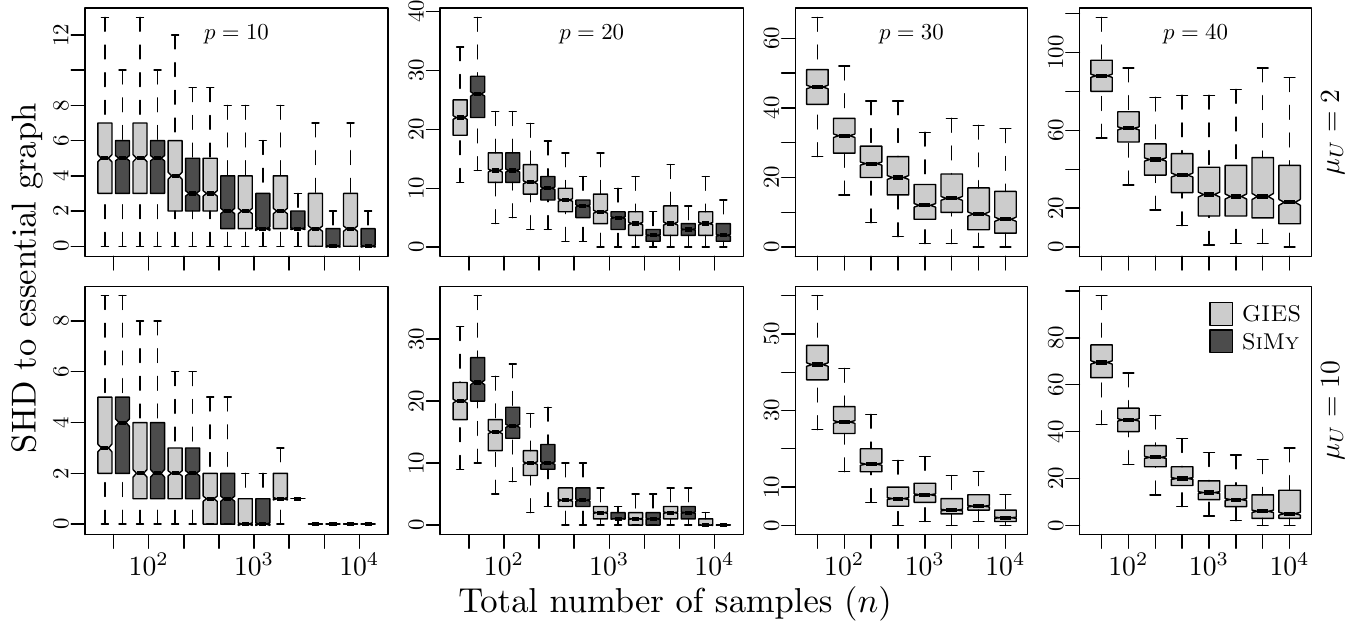}
  \end{center}
  \caption{SHD between estimated and true interventional essential graph as a function of the sample size $n$ for different numbers of variables $p$.  In each simulation, $p$ interventional data points were used, $2$ replicates for $p/2$ single-vertex intervention targets.  Interventions were performed with an expectation value of $\mu_U = 2$ (upper row) and $\mu_U = 10$ (lower row), respectively.}
  \label{fig:shdess-n}
\end{figure}

Figure \ref{fig:shdess-mu} supports our conjecture stated after Theorem \ref{thm:consistency}: even with few interventional data points (a total of $p$ in our case, compared to $n - p \gg p$ for $n = 1000$), the estimates of the causal models are substantially improved by only increasing the mean intervention values $\mu_U$.  However, for $p = 10$, this effect is not clearly visible.

\begin{figure}
  \begin{center}
    \includegraphics{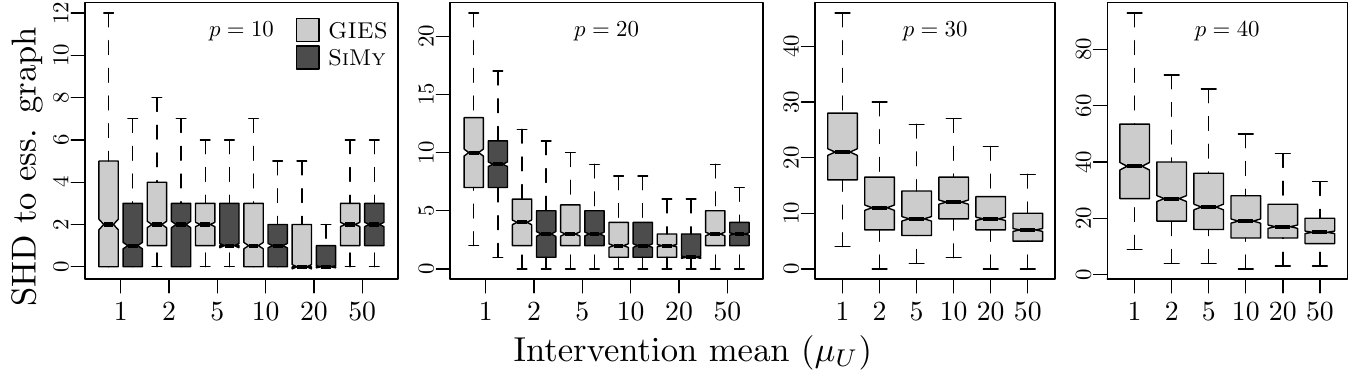}
  \end{center}
  \caption{SHD between estimated and true interventional essential graph as a function of the intervention mean $\mu_U$.  Results for simulations with a total sample size of $n = 1000$, of which $p$ data points originate from interventions at $20\%$ of the vertices.}
  \label{fig:shdess-mu}
\end{figure}

\section{Conclusions}

We have proposed a likelihood framework for joint modeling of Gaussian interventional and observational data. Such kind of data arises in many applications, notably in biology with measurements of wild-type individuals and modifications arising from interventional knock-outs of some genes.  Our likelihood approach has various interesting aspects which we summarize as follows. The parameters in the model are the observational directed acyclic graph (DAG) $D$ and the corresponding edge weights $B$ and error variances $\{\sigma^2_i\}_i$ (or instead of $B$ and $\{\sigma_i^2\}_i$ the corresponding covariance matrix of a Gaussian distribution). These parameters are global in the sense that every
intervention distribution is determined by these parameters via the do-calculus: in particular, this implies that only one or a few data per intervention suffice for reasonably accurate estimation since the corresponding distributions are all linked to the global parameters.

We show here that the BIC is consistent for estimating the corresponding interventional Markov equivalence class. The proof is rather involved since the various intervention distributions are not identical and do not easily fit into a standard setting. The interventional Markov equivalence class is an interesting and realistic target: it is smaller than the standard observational Markov equivalence class and it leads to a higher degree of identifiability when intervening at several variables. This has direct implications for tighter bounds for inferring causal effects \citep{Maathuis2009Estimating}. 

Besides the methodological development and theoretical derivations, we
present empirical results for real and simulated data. 

\section{Derivations and proofs}
\label{sec:proofs}

This section contains all proofs left out in earlier sections, namely the derivation of the maximum likelihood estimator for a given DAG (Section \ref{sec:proofs-mle}, proving results of Section \ref{sec:mle-given-dag}), and the proof of the consistency result for model selection (Section \ref{sec:proof-consistency} proving Theorem \ref{thm:consistency}).

\subsection{Explicit form of maximum likelihood estimator when DAG is known}
\label{sec:proofs-mle}

Gaussian densities form an exponential family.  The joint density of Gaussian random variables with expectation $\mu$ and covariance $\Sigma$ can be written as
\begin{equation}
  \textstyle f_{\mathcal{N}}(x; K, \nu) = (2\pi)^{-1/2} \exp \left[ \langle -\frac{1}{2} x x^\transp, K \rangle_{\mathcal{S}^p} + \langle x, \nu \rangle_{\sR^p} - \frac{1}{2}(\nu^\transp K^{-1}\nu - \log \det K) \right]\ ,
  \label{eqn:gaussian-exponential-family}
\end{equation}
where the inverse covariance matrix or precision matrix $K := \Sigma^{-1}$ and the transformed expectation value $\nu := K \mu$ form the natural parameters.  In (\ref{eqn:gaussian-exponential-family}), $\langle \smallbullet, \smallbullet \rangle_{\sR^p}$ stands for the canonical inner product on $\sR^p$, and $\langle \smallbullet, \smallbullet \rangle_{\mathcal{S}^p}$ denotes the inner product $\langle A, B \rangle_{\mathcal{S}^p} := \tr(AB)$ on the vector space $\mathcal{S}^p$ of symmetric $p \times p$ matrices.

The canonical form (\ref{eqn:gaussian-exponential-family}) of the exponential family of Gaussian distributions eases calculations with the interventional distributions (\ref{eqn:int-dist}), especially for our goal to derive a maximum likelihood estimator for a causal model with interventional data originating from \emph{different} interventions.  We hence start by calculating the natural parameters for the interventional distribution (\ref{eqn:int-dist}).  To simplify later calculations, we use the inverse error variances $\gamma_k := \sigma_k^{-2}$ to parameterize a Gaussian causal model from here on, together with the vector notation $\gamma := (\gamma_1, \ldots, \gamma_p)$.

\begin{lemma}
  \label{lem:natural-parameters}
  Let $\mu^{(I)}$ and $\Sigma^{(I)}$ be the expectation and covariance of the interventional distribution (\ref{eqn:int-dist}), respectively.  Then the following identities hold:
  \begin{align*}
    \textstyle K^{(I)} := \left(\Sigma^{(I)}\right)^{-1} & = (\mId - B)^\transp R^{(I)} \diag(\gamma) R^{(I)} (\mId - B) + Q^{(I)\transp} \tilde{K}^{(I)} Q^{(I)}\ , \\
    \nu^{(I)} := K^{(I)} \mu^{(I)} & = Q^{(I)\transp} \tilde{K}^{(I)} \mu_{U_I}\ , \\
    \textstyle \nu^{(I)\transp} \left( K^{(I)} \right)^{-1} \nu^{(I)} & = \mu_{U_I}^\transp \tilde{K}^{(I)} \mu_{U_I}\ , \\
    \log \det K^{(I)} & \textstyle = \sum_{k \notin I} \log \gamma_k + \log \det \tilde{K}^{(I)}\ .
  \end{align*}
  We make use here of the notation $\tilde{K}^{(I)} := \big( \tilde{\Sigma}^{(I)} \big)^{-1}$; $\tilde{\Sigma}^{(I)} := \diag(\tau^2_I)$ is the covariance matrix of the intervention variable $U_I$.
\end{lemma}

\begin{proof}
  To prove the formulae, we use the following identities of the auxiliary matrices (\ref{eqn:help-matrices}):
  \begin{alignat}{3}
    P^{(I)} P^{(I)\transp} & = \mId & \qquad P^{(I)} Q^{(I)\transp} & = 0 & \qquad Q^{(I)\transp} Q^{(I)} & = \mId - R^{(I)} \label{eqn:help-matrix-identities} \\
    Q^{(I)} Q^{(I)\transp} & = \mId & \qquad Q^{(I)} P^{(I)\transp} & = 0 & \qquad R^{(I)} R^{(I)} & = R^{(I)} \nonumber
  \end{alignat}
  To verify the claimed formula for the precision matrix $K^{(I)}$, it can be easily checked using the identities (\ref{eqn:help-matrix-identities}) that
  \begin{equation}
    \left[ R^{(I)} \diag(\sigma^2) R^{(I)} + Q^{(I)\transp} \tilde{\Sigma}^{(I)} Q^{(I)} \right]^{-1} = R^{(I)} \diag(\gamma) R^{(I)} + Q^{(I)\transp} \tilde{K}^{(I)} Q^{(I)}\ .
    \label{eqn:partial-inverse}
  \end{equation}
  We then find
  \begin{align*}
    K^{(I)} & \stackrel{(\ref{eqn:int-dist})}{=} \left(\mId - R^{(I) B}\right)^\transp \left[ R^{(I)} \diag(\gamma) R^{(I)} + Q^{(I)\transp} \tilde{K}^{(I)} Q^{(I)} \right] \left( \mId - R^{(I)} B \right)\\
    & \stackrel{\phantom{(11)}}{=} (\mId - B)^\transp R^{(I)} \diag(\gamma) R^{(I)} (\mId - B) + Q^{(I)\transp} \tilde{K}^{(I)} Q^{(I)}\ ,
  \end{align*}
  where we again use several of the identities (\ref{eqn:help-matrix-identities}) in the last step.
  
  By making use of equations (\ref{eqn:int-dist}) and (\ref{eqn:partial-inverse}) again, we can calculate the transformed expectation:
  \begin{align*}
    \nu^{(I)} & = \left( \mId - R^{(I)} R \right)^\transp \left[ R^{(I)} \diag(\gamma) R^{(I)} + Q^{(I)\transp} \tilde{K}^{(I)} Q^{(I)} \right] Q^{(I)\transp} \mu_{U_I} \\
    & = Q^{(I)\transp} \tilde{K}^{(I)} \mu_{U_I} \ ;
  \end{align*}
  the last step is again a consequence of the identities (\ref{eqn:help-matrix-identities}).
  
  For the next formula, we use the fact that $B$ is a nilpotent matrix; it is not hard to see that every matrix satisfying the DAG constraint actually is nilpotent.  Therefore the inverse of $\left(\mId - R^{(I)} B \right)$ can be calculated as $\left( \mId - R^{(I)} B \right)^{-1} = \sum_{k = 0}^{p - 1} \left(R^{(I)} B \right)^k$.  Together with the identities (\ref{eqn:help-matrix-identities}) and the representation of $\mu^{(I)}$ in (\ref{eqn:int-dist}), we conclude that
  $$
    Q^{(I)} \mu^{(I)} = \sum_{k = 0}^{p - 1} Q^{(I)} \left(R^{(I)} B \right)^k Q^{(I)\transp} \mu_{U_I} = \mu_{U_I} \ .
  $$
  It follows that
  $$
    \nu^{(I)\transp} \left( K^{(I)} \right)^{-1} \nu^{(I)} = \mu^{(I)\transp} \nu^{(I)} = \mu^{(I)\transp} Q^{(I)\transp} \tilde{K}^{(I)} \mu_{U_I} = \mu_{U_I}^\transp \tilde{K}^{(I)} \mu_{U_I} \ ,
  $$
  where we used the formula for $\nu^{(I)}$ already proven before.
  
  To calculate the determinant of $K^{(I)}$ finally, note that there is a permutation matrix $P$ such that
  $$
    P \left[ R^{(I)} \diag(\gamma) R^{(I)} + Q^{(I)\transp} \tilde{K}^{(I)} Q^{(I)} \right] P^\transp
  $$
  is a block matrix.  Hence
  $$
      \det K^{(I)} = \det \tilde{K}^{(I)} \cdot \prod_{k \notin I} \gamma_k
  $$
  or $\log \det K^{(I)} = \sum_{k \notin I} \log \gamma_k + \log \det \tilde{K}^{(I)}$, which completes the proof.
\end{proof}

Up to now, we have only considered a single interventional distribution.  In the next lemma, we provide a formula for the likelihood of an interventional dataset originating from \emph{multiple} intervention targets as defined in (\ref{eqn:structural-eq2}).  In the following, we simplify notation by unifying observational and interventional data point in a common framework.  For this aim, we reuse the convention at the end of Section \ref{sec:mle-given-dag} and consider the \emph{entire} data set $(X^{(i)})_{i=1}^n$, $n = \nobs + \nint$, of all observational and interventional data points.  To make notation short, we denote the complete data set by the matrix $\mathbf{X}$, having the rows $X^{(1)}, \ldots, X^{(n)}$, and the list of intervention targets $T^{(1)}, \ldots, T^{(n)}$ by $\mathcal{T}$.  Recall that an observational data point $X^{(i)}$ is marked by the empty target $T^{(i)} = \emptyset$.

\begin{lemma}
  \label{lem:likelihood}
  Let $(\mathcal{T}, \mathbf{X})$ be an interventional dataset as defined above, produced by a Gaussian causal model with structure $D$.  Moreover, let $B \in \mathbf{B}(D)$ be a weight matrix and $\gamma \in \sR_{> 0}^p$ a vector of inverse error variances.  Denote by $n^{(I)} := |\{i \spst T^{(i)} = I\}|$ and $S^{(I)} := \frac{1}{n^{(I)}} \sum_{i: T^{(i)} = I} X^{(i)} X^{(i)\transp}$ (empirical covariance matrix for intervention $I \in \mathcal{I}$).  Then the log-likelihood of $(\mathcal{T}, \mathbf{X})$ given parameters $B$ and $\gamma$ is
  \begin{multline*}
    \ell_D(B, \gamma; \mathcal{T}, \mathbf{X}) = - \frac{1}{2} \sum_{I \in \mathcal{I}} n^{(I)} \tr \left( S^{(I)} K^{(I)} \right) + \frac{1}{2} \sum_{I \in \mathcal{I}} n^{(I)} \log \det K^{(I)} + C \\
    = - \frac{1}{2} \sum_{I \in \mathcal{I}} n^{(I)} \tr \left[ S^{(I)} (\mId - B)^\transp R^{(I)} \diag(\gamma) R^{(I)} (\mId - B) \right] + \frac{1}{2} \sum_{I \in \mathcal{I}} n^{(I)} \sum_{k \notin I} \log \gamma_k + C' \ ,
  \end{multline*}
  where $C$ and $C'$ are constants given by the dataset $(\mathcal{T}, \mathbf{X})$ that do not depend on the model parameters $B$ and $\gamma$.
\end{lemma}

Note that in the case of purely observational data (that is, if $T^{(i)} = \emptyset$ for all $i$), this result reproduces the classical log-likelihood \citep[see, for example,][]{Banerjee2008Model}
$$
  2 \ell_D(B, \gamma; (\emptyset)_{i = 1}^n, \mathbf{X}) = n (\log \det K - \tr(S K)) + C\ .
$$

\begin{proof}
  The likelihood of the entire data set is the product of the sample likelihoods (\ref{eqn:int-dist}):
  \begin{align*}
    \ell_D(B, \gamma; \mathcal{T}, \mathbf{X}) & = \sum_{i=1}^n \log f\left(X^{(i)} \spst \doop(X_{T^{(i)}}^{(i)} = U^{(i)}_{T^{(i)}}) \right) \\
    & \stackrel{(\ref{eqn:int-dist})}{=} \sum_{i=1}^n \log f_{\mathcal{N}}\left(X^{(i)}; K^{(T^{(i)})}, \nu^{(T^{(i)})}\right) \\
    & \stackrel{(\ref{eqn:gaussian-exponential-family})}{=} - \frac{1}{2} \sum_{i=1}^n \tr\left( X^{(i)} X^{(i)\transp} K^{(T^{(i)})} \right) + \frac{1}{2} \sum_{i=1}^n \log \det K^{(T^{(i)})}  + C \\
    & = - \frac{1}{2} \sum_{I \in \mcI} n^{(I)} \tr\left(S^{(I)} K^{(I)} \right) + \frac{1}{2} \sum_{I \in \mcI} n^{(I)} \log \det K^{(I)} + C \ .
  \end{align*}
  In the calculations above, $C$ stands for a constant that is independent of the model parameters $B$ and $\gamma$ (note that, by Lemma \ref{lem:natural-parameters}, the remaining terms from Equation (\ref{eqn:gaussian-exponential-family}) are independent of model parameters).
  
  The second line of the lemma follows easily from the first one by applying the identities given in Lemma \ref{lem:natural-parameters}.
\end{proof}

The following lemma shows that the log-likelihood derived before is \emph{decomposable} \citep{Chickering2002Optimal} in the sense that it can be written as a sum of terms that only depend on a vertex and its parents.

\begin{lemma}
  \label{lem:likelihood-decomposition}
  Using the definitions $n^{(-k)} := \sum_{I \in \mathcal{I}: k \notin I} n^{(I)}$ and $S^{(-k)} := \sum_{I \in \mathcal{I}: k \notin I} \frac{n^{(I)}}{n^{(-k)}} S^{(I)}$, the log-likelihood of Lemma \ref{lem:likelihood} can be decomposed as follows:
  \begin{align*}
    \ell_D(B, \gamma; \mathcal{T}, \mathbf{X}) & = \sum_{k = 1}^p \ell_k (B_{k \smallbullet}, \gamma_k; \mathcal{T}, \mathbf{X}) + C, \\
    \ell_k (B_{k \smallbullet}, \gamma_k; \mathcal{T}, \mathbf{X}) & = -\frac{1}{2} n^{(-k)} \Big[ \gamma_k (\mId - B)_{k \smallbullet} S^{(-k)} \left( (\mId - B)_{k \smallbullet}\right)^\transp - \log \gamma_k \Big],
  \end{align*}
  where $C$ is a constant that does not depend on the parameters $\gamma$ and $B$.  The calculation of the partial likelihoods $\ell_k$ only involves data measured at vertex $k$ and its parents $\pa(k)$.
\end{lemma}

\begin{proof}
  The decomposition of the second summand in Lemma \ref{lem:likelihood} is easy to verify:
  $$
    \sum_{I \in \mathcal{I}} n^{(I)} \sum_{k \notin I} \log \gamma_k = \sum_{i = 1}^n \sum_{k \notin T^{(i)}} \log \gamma_k = \sum_{k = 1}^p \sum_{i: k \notin T^{(i)}} \log \gamma_k = \sum_{k = 1}^p n^{(-k)} \log \gamma_k.
  $$
  
  The decomposition of the first summand makes use of the fact that $\tr(AB) = \tr(BA)$ for any matrices $A$ and $B$ for which $AB$ and $BA$ are defined:
  \begin{align*}
    \sum_{I \in \mathcal{I}} n^{(I)} \tr & \; \Big[ S^{(I)} (\mId - B)^\transp R^{(I)} \diag(\gamma) R^{(I)} (\mId - B) \Big] \\
    & = \sum_{i = 1}^n \tr \Big[ R^{(T^{(i)})} \diag(\gamma) R^{(T^{(i)})} (\mId - B) X^{(i)} X^{(i)\transp} (\mId - B)^\transp \Big] \\
    & = \sum_{i = 1}^n \sum_{k \notin T^{(i)}} \gamma_k (\mId - B)_{k \smallbullet} X^{(i)} X^{(i)\transp} \left((\mId - B)_{k \smallbullet}\right)^\transp \\
    & = \sum_{k = 1}^p n^{(-k)} \gamma_k (\mId - B)_{k \smallbullet}  S^{(-k)} \left((\mId - B)_{k \smallbullet}\right)^\transp.
  \end{align*}

  The $k\subscr{th}$ column of $\mId - B$, $(\mId - B)_{k \smallbullet}$ only has entries at indices $\{k\} \cup \pa(k)$, so the calculation only includes rows and columns of the empirical covariance matrix with those indices and hence only uses data from vertex $k$ and its parents.
\end{proof}

Lemma \ref{lem:likelihood-decomposition} shows that, for a fixed DAG $D$, the maximum likelihood estimates for the weight matrix and the error variances can be calculated ``locally'', that is only involving data of single vertices and their parents.

\begin{lemma}
  \label{lem:maximum-likelihood}
  For a fixed DAG $D$ and given data, the maximum likelihood estimate for its parameters $\sigma$ and $B$ are
  $$
    \hat{B}_{k, \pa(k)} = S^{(-k)}_{k, \pa(k)} \left( S^{(-k)}_{\pa(k), \pa(k)} \right)^{-1}, \quad \hat{\sigma}_k^2 = (\mId - \hat{B})_{k \smallbullet} S^{(-k)} \left((\mId - \hat{B})_{k \smallbullet}\right)^\transp,
  $$
  The maximum partial likelihoods are
  \begin{align*}
    \sup_{B_{k \smallbullet}, \gamma_k} \ell_k(B_{k \smallbullet}, \gamma_k; \mathcal{T}, \mathbf{X}) & = - \frac{1}{2} n^{(-k)} \left(1 + \log \hat{\sigma}_k^2 \right) \\
    & = - \frac{1}{2} n^{(-k)} \Big\{ 1 + \log \Big[ S^{(-k)}_{kk} - S^{(-k)}_{k, \pa(k)} \Big( S^{(-k)}_{\pa(k), \pa(k)} \Big)^{-1} S^{(-k)}_{\pa(k), k} \Big] \Big\}
  \end{align*}
\end{lemma}

\begin{proof}
  The maximum likelihood estimate must be a root of the derivative of the likelihood.  From Lemma \ref{lem:likelihood-decomposition}, we see that
  $\partder{}{B_{ki}}\ell = \partder{}{B_{ki}}\ell_k$ for $i = 1, \ldots, p$.  This partial derivative is
  \begin{equation}
    \partder{}{B_{ki}} \ell_k(B_{k \smallbullet}, \gamma_k; \mathcal{T}, \mathbf{X}) \propto (\mId - B)_{k \smallbullet} S^{(-k)}_{\smallbullet i} = S^{(-k)}_{ki} - B_{k \smallbullet} S^{(-k)}_{\smallbullet i} \ .
    \label{eqn:partder-weight-matrix}
  \end{equation}
  For a fixed $k$, $B_{k \smallbullet}$ has one non-zero entry for every parent of $k$ in the DAG $D$.  For those entries, we get the system of linear equations
  $$
    B_{k, \pa(k)} S^{(-k)}_{\pa(k), i} = S^{(-k)}_{ki}, \quad \spforall i \in \pa(k),
  $$
  by setting the partial derivatives \eqref{eqn:partder-weight-matrix} to zero.  In matrix notation, this reads
  $$
    B_{k, \pa(k)} S^{(-k)}_{\pa(k), \pa(k)} = S^{(-k)}_{k, \pa(k)}
  $$
  and has the solution
  $$
    \hat{B}_{k, \pa(k)} = S^{(-k)}_{k, \pa(k)} \left( S^{(-k)}_{\pa(k), \pa(k)} \right)^{-1}\ ;
  $$
  note that $S^{(-k)}_{k, \pa(k)}$ is invertible almost surely if $n^{(-k)} > |\pa(k)|$.
  
  The derivative with respect to the error variances is
  $$
    \partder{}{\gamma_k} \ell_k(B_{k \smallbullet}, \gamma_k) \propto (\mId - B)_{k \smallbullet} S^{(-k)} \left((\mId - B)_{k \smallbullet}\right)^\transp - \frac{1}{\gamma_k}
  $$
  and has the inverse root
  \begin{align*}
    \frac{1}{\hat{\gamma}_k} & = \hat{\sigma}_k^2 = (\mId - \hat{B})_{k \smallbullet} S^{(-k)} \left((\mId - \hat{B})_{k \smallbullet}\right)^\transp \\
    & = S^{(-k)}_{kk} - S^{(-k)}_{k, \pa(k)} \Big( S^{(-k)}_{\pa(k), \pa(k)} \Big)^{-1} S^{(-k)}_{\pa(k), k}.
  \end{align*}
  By plugging this into the formula of Lemma \ref{lem:likelihood-decomposition}, we immediately find the formula for the supremum of the partial likelihoods.
\end{proof}

\subsection{Definition of interventional Markov equivalence class}
\label{sec:def-int-equivalence-class}

The observational Markov equivalence class of a DAG can be described as follows. For a DAG $D$, denote by ${\cal M}(D) = \{f;\ f\ \mbox{Markov with respect to}\ D\}$ all distributions which are Markov with respect to $D$. Thereby, the Markovian property is meant to be the factorization property as in (\ref{eqn:factorization}), and we denote by $f$ the density of the $p$-dimensional Gaussian distribution. Two DAGs $D \sim D'$ are Markov equivalent, if and only if ${\cal M}(D) = {\cal M}(D')$. The observational equivalence class of a DAG $D$ is then denoted by ${\cal E}(D)$ which can be represented as an essential graph which is a chain graph with directed and undirected edges \citep{Andersson1997Characterization}. 

For the interventional Markov equivalence class, we proceed as follows. For a DAG $D$, consider the corresponding intervention DAG $D_I$ where we remove all edges which point from $\pa(I)$ to $I$. Furthermore, consider a family of intervention targets ${\cal I}$ and corresponding tuples of densities $(f_I)_{I \in {\cal I}}$, where each element corresponds to an intervention target $I \in {\cal I}$. Let
\begin{eqnarray*}
  {\cal M}_{\cal I}(D) = \{(f_I)_{I \in {\cal I}};& &\spforall I \in {\cal I}:\
  f_I \in {\cal M}(D_I),\ \mbox{and}\\
  & &\spforall I,J \in {\cal I}, \spforall i
  \notin I \cup J:\ f_I(x_i|x_{\pa_D(i)}) = f_J(x_i|x_{\pa_D(i)})\}.
\end{eqnarray*}
Two DAGs $D$ and $D'$ are interventionally Markov equivalent with respect to the family of targets $\mcI$ (notation: $D \sim_\mcI D'$) if and only if ${\cal M}_{\cal I}(D) = {\cal M}_{\cal I}(D')$ \citep{Hauser2012Characterization}. For a DAG $D$, the interventional Markov equivalence class with respect to \mcI{} (or \mcI-Markov equivalence class) is denoted by $[D]_\mcI$ which, as in the observational case, can be characterized by an essential graph $\mathcal{E_I}(D)$ \citep{Hauser2012Characterization}. For ${\cal I} = \emptyset$, the definition coincides with the observational Markov equivalence class above. Although the definition of interventional Markov equivalence is somewhat cumbersome, the defined object indeed represents the DAGs which are equivalent and non-distinguishable from the interventional distributions (and if ${\cal   I}$ also contains the $\emptyset$-target, from observational and interventional distributions). In other words, assuming faithfulness as in (\ref{eqn:pobs}), the ${\cal I}$ interventional Markov equivalence is identifiable from the distributions. 

\subsection{Proof of Theorem \ref{thm:consistency}}
\label{sec:proof-consistency}

In the previous section, we calculated the maximum of the likelihood of causals models given a set of interventional data $(\mathcal{T}, \mathbf{X})$.  For model selection, that is, estimating the causal model that produced a given dataset, the model complexity has to be penalized to avoid overfitting.  For large interventional (and potentially observational) samples, it stands to reason to choose the complexity penalty of the Bayesian information criterion (BIC).

The maximization of the BIC of a growing sequence of i.i.d.\ data is known to lead to consistent model selection from a set of curved exponential models \citep{Haughton1988Choice}.

\begin{definition}[Curved exponential model; \citealp{Haughton1988Choice}]
  \label{def:curved-exponential-model}
  Let $\mathcal{P} = \{f(x; \theta) = h(x) \exp[ \langle T(x), \theta \rangle - b(\theta) ] \spst \theta \in \Theta\}$ be an exponential family with natural parameter space $\Theta \subset \sR^k$.  A curved exponential model is a set of parameters of the form $M \cap \Theta$, where $M$ is a smooth connected manifold embedded in $\sR^k$.
\end{definition}

Suppose $(X^{(i)})_{i=1}^n$ is a sequence of i.i.d.\ realizations from a density in the exponential family of Definition \ref{def:curved-exponential-model}, and let $M \cap \Theta$ be an curved exponential model in that family.  The \emph{Bayesian information criterion} or \emph{BIC} of $M \cap \Theta$ is then defined as
\begin{align}
  S(M; \mathbf{X}) & := \sup_{\theta \in M \cap \Theta} \log \prod_{i = 1}^n f(X^{(i)}; \theta) - \frac{1}{2} \dim(M) \log n \nonumber \\
  & = n \sup_{\theta \in M \cap \Theta} (\langle \overline{T}_n, \theta \rangle - b(\theta)) - \frac{1}{2} \dim(M) \log n \ , \label{eqn:bic}
\end{align}
where $\overline{T}_n$ stands for the mean statistic $\overline{T}_n := \frac{1}{n} \sum_{i=1}^n T(X^{(i)})$ and $\mathbf{X}$ is the data matrix having the samples $X^{(i)}$ as rows.

\begin{theorem}[Consistency of the BIC; \citealp{Haughton1988Choice}]
  \label{thm:bic-consistency}
  Let $M_1 \cap \Theta, M_2 \cap \Theta, \ldots$ be a finite set of curved exponential models in the natural parameter space $\Theta$ of an exponential family as in Definition \ref{def:curved-exponential-model} with the following property: for each $i \neq j$, if a point in $\overline{M_i}$ is in $M_j \cap \interior{\Theta}$, then it is in $M_i$.
  
  Assume $\theta \in \interior{\Theta}$ and let $M_i$ and $M_j$ be two different curved exponential models.  If $\theta \in M_i \setminus M_j$, or if $\theta \in M_i \cap M_j$ with $\dim(M_i) < \dim(M_j)$, then
  $$
    \lim_{n \to \infty} P_\theta[S(M_i; \mathbf{X}) > S(M_j; \mathbf{X})] = 1\ .
  $$
\end{theorem}

As we explained in Section \ref{sec:structure-learning-bic}, we regard the intervention targets $T^{(1)}, \ldots, T^{(n)}$ as a \emph{random} sequence, taking a ``value'' $I \in \mcI$ with probability $w_I$ (assumption (A3) of Section \ref{sec:structure-learning-bic}).  With this assumption, we can treat the complete data set $(T^{(i)}, X^{(i)})_{i = 1}^n$ as i.i.d.\ realizations of random variables $(T, X) \in \mcI \times \sR^p$.  Expressed in this notation, we have shown in Section \ref{sec:proofs-mle} that the \emph{conditional} densities $f(x \spst T = I) = f_\intv(x \spst \doop(X_I = U_I))$ belong to an exponential family.  In the next proposition, we show that also the joint density of $(T, X)$ belongs to an exponential family.

\begin{proposition}
  \label{prop:mixing-exponential-family}
  Consider a set of random variables $(X, Y) \in \sR^p \times \{1, \ldots, J\}$ with $P[Y = j] = w_j$, $1 \leq j \leq J$, and $(X \spst Y = j) \sim f(\cdot \ ; \theta_j)$, where $f(x; \theta)$ is a density from an exponential family:
  $$
    f(x; \theta) = h(x) \exp \left[ \langle T(x), \theta \rangle - b(\theta) \right] \ .
  $$
  Then the joint density of $X$ and $Y$ is also an element of an exponential family, namely
  $$
    f(x, y; \boldsymbol{\theta}, \eta) = h(x) \exp \left[ \left\langle S(x, y), \binom{\boldsymbol{\theta}}{\eta} \right\rangle - a(\boldsymbol{\theta}, \eta) \right]\ .
  $$
  The natural parameters are given by $\boldsymbol{\theta} = (\theta_1^\transp, \ldots, \theta_J^\transp)^\transp$ and $\eta = (\eta_1, \ldots, \eta_{J - 1})^\transp$ with $\eta_j = \log \frac{w_j}{w_J} - b(\theta_j) + b(\theta_J)$.  The sufficient statistic $S$ and the log-partion function $a$ are given by
  \begin{align*}
    S(x, y) & = (\delta_{y,1} T(x)^\transp, \ldots, \delta_{y,J} T(x)^\transp, \delta_{y,1}, \ldots, \delta_{y,J-1})^\transp\ , \\
    a(\boldsymbol{\theta}, \eta) & = b(\theta_j) + \log \left[1 + \sum_{j=1}^{J-1} \exp \left( \eta_j + b(\theta_j) - b(\theta_J) \right) \right]\ .
  \end{align*}
\end{proposition}

\begin{proof}
  A straight-forward calculation yields to the claimed result:
  \begin{align*}
    f(x, y; \boldsymbol{\theta}, \eta) & = w_y f(x; \theta_y) \\
    & = h(x) \exp[ \langle T(x), \theta_y \rangle - b(\theta_y) + \log w_y ] \\
    & = h(x) \exp\left[ \sum_{j=1}^J \langle \delta_{y,j} T(x), \theta_j \rangle - \sum_{j=1}^J \delta_{y,j} (\log w_j - b(\theta_j)) \right] \\
    & = h(x) \exp\left[ \sum_{j=1}^J \langle \delta_{y,j} T(x), \theta_j \rangle - \sum_{j=1}^{J-1} \delta_{y,j} \left(\log \frac{w_j}{w_J} - b(\theta_j) + b(\theta_J) \right) \right. \\
    & \phantom{= h(x) \exp[} \left. + \left(1 - \sum_{j=1}^{J-1} \delta_{j,y}\right) (\log w_J - b(\theta_J)) + \sum_{j=1}^{J-1} \delta_{j,y} (\log w_J - b(\theta_J)) \right] \\
    & = h(x) \exp\left[ \sum_{j=1}^J \langle \delta_{y,j} T(x), \theta_j \rangle - \sum_{j=1}^{J-1} \delta_{y,j} \left(\log \frac{w_j}{w_J} - b(\theta_j) + b(\theta_J) \right) \right. \\
    & \phantom{= h(x) \exp[} \left. + \log w_J - b(\theta_J) \right] \\
    & = h(x) \exp \left[ \left\langle S(x, y), \binom{\boldsymbol{\theta}}{\eta} \right\rangle + \log w_J - b(\theta_J) \right]
  \end{align*}
  with the definitions of $S(x, y)$, $\boldsymbol{\theta}$ and $\eta$ from above.
  
  To finish the calculation, we need to express $w_J$ as a function of $\boldsymbol{\theta}$ and $\eta$: since
  $$
    w_J = 1 - \sum_{j=1}^{J-1} w_j = 1 - \sum_{j=1}^{J-1} \exp[ \eta_j + b(\theta_j) - b(\theta_J)] w_J\ ,
  $$
  we find
  $$
    w_J = \left[ 1 + \sum_{j=1}^{J-1} \exp( \eta_j + b(\theta_j) - b(\theta_J)) \right]^{-1}\ ,
  $$
  what immediately yields the claimed log-partition function $a(\boldsymbol{\theta}, \eta)$.
\end{proof}

In order to prove the consistency of the BIC for causal model selection under interventions in the limit of large interventional samples, we must show that the models described by different DAGs fit the prerequisites of Theorem \ref{thm:bic-consistency}.

We have already seen that a single Gaussian interventional density (\ref{eqn:int-dist}) is a representative of an exponential family with natural parameters $K^{(I)}$ and $\nu^{(I)}$ living in $\mathcal{S}^p$ and $\sR^p$, respectively (see (\ref{eqn:gaussian-exponential-family})).  By Proposition \ref{prop:mixing-exponential-family}, we conclude that the natural parameter space for the complete family of interventions is
$$
  \underbrace{( \mathcal{S}^p_{>0} )^J}_{=: \mathcal{S}} \times \underbrace{(\sR^p)^J}_{=: \mathcal{V}} \times \underbrace{\sR^{J - 1}}_{=: \mathcal{W}} \ ,
$$
where we write $J := |\mcI|$.  We have already seen that the interventional densities are determined by \emph{model parameters} and \emph{experimental parameters}; the model parameters are $B \in \mathbf{B}(D)$ and $\gamma \in \sR^p_{>0}$.  Therefore the sets of natural parameters corresponding to different models are parameterized by functions
$$
  \Phi^\mcI_D: \mathbf{B}(D) \times \sR^p_{>0} \to \mathcal{S} \times \mathcal{V} \times \mathcal{W} \ .
$$
Before showing that the images of those maps form indeed a set of embedded manifolds in $\mathcal{S} \times \mathcal{V} \times \mathcal{W}$ satisfying the prerequisites of Theorem \ref{thm:bic-consistency}, we sum up our notation from above and from Lemma \ref{lem:natural-parameters}.

\begin{definition}
  \label{def:exponential-parameterization}
  Let $D$ be a DAG.  Furthermore, let $\mcI$ be a conservative family of intervention targets, and $T \in \mcI$ arbitrary.  Then we define
  \begin{align*}
    \Phi^\mcI_D: \ \mathbf{B}(D) \times \sR^p_{>0} & \to \mathcal{S} \times \mathcal{V} \times \mathcal{W}, \\
    (B, \gamma) & \mapsto \left( \big( K^{(I)}(B, \gamma) \big)_{I \in \mcI}, \big( \nu^{(I)} \big)_{I \in \mcI}, \big( \eta^{(I)} \big)_{I \in \mcI \setminus \{T\}} \right)
  \end{align*}
  with
  \begin{align*}
    K^{(I)}(B, \gamma) & = (\mId - B)^\transp R^{(I)} \diag(\gamma) R^{(I)} (\mId - B)^\transp + Q^{(I)\transp} \tilde{K}^{(I)} Q^{(I)}\ , \\
    \nu^{(I)} & = Q^{(I)\transp} \tilde{K}^{(I)} \tilde{\mu}^{(I)}\ , \\
    \eta^{(I)} & = \log \frac{\tilde{w}_I}{\tilde{w}_T} - b\big[K^{(I)}(B, \gamma), \nu^{(I)}\big] + b\big[K^{(T)}(B, \gamma), \nu^{(T)}\big] \ , \\
    b(K, \nu) & = \frac{1}{2}(\nu^\transp K^{-1} \nu - \log \det K) \ .
  \end{align*}
  Furthermore, we denote the image of $\Phi^\mcI_D$ in $\mathcal{S} \times \mathcal{V} \times \mathcal{W}$ by $M^\mcI_D$.
\end{definition}

\begin{proposition}
  \label{prop:smooth-embedding}
  With the notation from Definition \ref{def:exponential-parameterization}, the image $M^\mcI_D$ is an embedded, smooth manifold in $\mathcal{S} \times \mathcal{V} \times \mathcal{W}$.
\end{proposition}

\begin{proof}
  We have to prove the following points:
  \begin{subprop}
    \item \label{itm:smooth} $\Phi^\mcI_D$ is smooth;
    \item \label{itm:injective} $\Phi^\mcI_D$ is injective (and hence a bijection onto its image);
    \item \label{itm:homeomorphic} the inverse of $\Phi^\mcI_D$ (on its image) is continuous;
    \item \label{itm:immersion} $\Phi^\mcI_D$ is an immersion, that is, its derivative is injective everywhere.
  \end{subprop}
  Points \ref{itm:injective} and \ref{itm:homeomorphic} say that $\Phi^\mcI_D$ is a \emph{topological embedding}; points \ref{itm:smooth} and \ref{itm:immersion} strengthen the result to show that $\Phi^\mcI_D$ is even an embedding in the sense of differential geometry.
  
  We will now give the (rather technical) proofs of the aforementioned four points.  Throughout the proofs, we will always assume w.l.o.g. that the vertices of $D = ([p], E)$ are numbered according to an inverse topological sorting, such that all matrices in $\mathbf{B}(D)$ are strictly lower triangular matrices.
  
  \begin{subprop}
    \item The smoothness of $\Phi^\mcI_D$ is immediately clear from its definition: $\Phi^\mcI_D$ is a composition of smooth functions.
        
    \item Let $(B, \gamma)$ and $(B', \gamma') \in \mathbf{B}(D) \times \sR^p_{>0}$ such that $\Phi^\mcI_D(B, \gamma) = \Phi^\mcI_D(B', \gamma')$; by the definition of $\Phi^\mcI_D$, this is the case if and only if $K^{(I)}(B, \gamma) = K^{(I)}(B', \gamma')$ for all $I \in \mcI$.  This condition simplifies to
    $$
      (\mathbbm{1} - B) R^{(I)} \diag(\gamma) R^{(I)} (\mathbbm{1} - B)^\transp = (\mathbbm{1} - B') R^{(I)} \diag(\gamma') R^{(I)} (\mathbbm{1} - B')^\transp \ ,
    $$
    or, with the abbreviation $A := (\mathbbm{1} - B)^{-1} (\mathbbm{1} - B')$,
    \begin{equation}
      R^{(I)} \diag(\gamma) R^{(I)} A^{-\transp} = A R^{(I)} \diag(\gamma') R^{(I)} \ . \label{eqn:different-parameterizations}
    \end{equation}
    By the assumption made before, $B$ and $B'$ are strict lower triangular matrices, hence $A$ is a lower triangular matrix with ones as diagonal entries.  Then, the left-hand side of equation (\ref{eqn:different-parameterizations}) is an \emph{upper} triangular matrix, whereas the right-hand side is a \emph{lower} triangular matrix.  We conclude that both sides of the equation must consist of a \emph{diagonal} matrix, and that we can transpose the left-hand side:
    \begin{equation}
      A^{-1} R^{(I)} \diag(\gamma) R^{(I)} = A R^{(I)} \diag(\gamma') R^{(I)} \ . \label{eqn:different-parameterizations-transp}
    \end{equation}
    For some $a \notin I$, the $a\supscr{th}$ column of the matrix equation (\ref{eqn:different-parameterizations-transp}) reads
    \begin{equation}
      (A^{-1} \diag(\gamma))_{\smallbullet \; a} = (A^{-1})_{\smallbullet \; a} \gamma_a = A_{\smallbullet \; a} \gamma'_a = (A \diag(\gamma'))_{\smallbullet \; a} \ . \label{eqn:different-parameterizations-column}
    \end{equation}
    Since the family of targets \mcI{} is \emph{conservative}, there is, for every $a \in [p]$, some $I \in \mcI$ such that $a \notin I$; because equation \ref{eqn:different-parameterizations-transp} holds for every $I \in \mcI$, the column-wise equation (\ref{eqn:different-parameterizations-column}) holds for every $a \in [p]$, so we finally find $A^{-1} \diag(\gamma) = A \diag(\gamma')$, or, equivalently, $A^2 = \diag(\gamma) \diag(\gamma')^{-1}$.  Because the diagonal of $A^2$ only consists of ones, we see that $\gamma = \gamma'$.  It follows that $A^2 = \mathbbm{1}$, and because $A$ is a unit triangular matrix, this means that $A = \mathbbm{1}$, and hence, by definition of $A$, $B = B'$.  Therefore, $\Phi^\mcI_D$ is injective.
    
    \item We can restrict our considerations to the parameterizations of the precision matrices:
    \begin{align}
      K^{(I)}_{I^c, I^c} & = P^{(I)} K^{(I)} P^{(I)\transp} = (\mathbbm{1} - B)_{I^c, I^c} \diag(\gamma_{I^c}) (\mathbbm{1} - B_{I^c, I^c})^\transp \ , \label{eqn:precision-non-I-non-I} \\
      K^{(I)}_{I, I^c} & = Q{(^I)} K^{(I)} P^{(I)\transp} = - Q^{(I)} B R^{(I)} \diag(\gamma) \big( P^{(I)\transp} - R^{(I)} B^\transp P^{(I)\transp} \big) \nonumber \\
      & = - B_{I, I^c} \diag(\gamma_{I^c}) \left(\mathbbm{1} - B_{I^c, I^c} \right)^\transp \ . \label{eqn:precision-I-non-I}
    \end{align}
    By assuming, as before, that $B$ is a strict lower triangular matrix, (\ref{eqn:precision-non-I-non-I}) represents the Cholesky decomposition of $K^{(I)}_{I^c, I^c}$.  This decomposition is unique, and $B_{I^c, I^c}$ as well as $\gamma_{I^c}$ depend \emph{continuously} on $K^{(I)}_{I^c, I^c}$ \citep{Schwarz2006Numerische}.
    
    For each $b \in [p]$, there is some $I \in \mcI$ that does not contain $b$ since \mcI{} is conservative.  Hence $\gamma_b$ can be calculated out of $K^{(I)}_{I^c, I^c}$ by performing the Cholesky decomposition as described above.  This shows that $\gamma$ is a continuous function of the precision matrices $(K^{(I)})_{I \in \mcI}$.
    
    Assume now that $a \grarright b$ is an arrow in $D$, and let $I \in \mcI$ be an intervention target with $b \notin I$.  If $a \notin I$, $B_{ab}$ can also be calculated from $K^{(I)}_{I^c, I^c}$ via the (continuous) Cholesky decomposition.  Otherwise, $B_{ab}$ is an entry of the matrix $B_{I, I^c}$ which can be calculated by solving equation (\ref{eqn:precision-I-non-I}):
    $$
      B_{I, I^c} = - K^{(I)}_{I, I^c} (\mathbbm{1} - B_{I^c, I^c})^{-\transp} \diag(\gamma_{I^c})^{-1} \ ,
    $$
    which is a \emph{continuous} function since the matrix inversion is continuous.  Altogether, also the matrix $B \in \mathbf{B}(D)$ is a continuous function of the precision matrices $(K^{(I)})_{I \in \mcI}$, what proves the claim.
    
    \item We have to show that the derivative $\incr{\Phi^\mcI_D}(B, \gamma)$ has maximal rank for all $(B, \gamma) \in \mathbf{B}(D) \times \sR^p_{>0}$.  For that aim, we consider the canonical basis
    $$
      \{(H^{(a, b)}, 0)\}_{(a, b) \in E} \cup \{(0, e_i)\}_{1 \leq i \leq p}
    $$
    of $\mathbf{B}(D) \times \sR^p$, the tangent space of $\mathbf{B}(D) \times \sR^p_{>0}$ at the point $(B, \gamma)$, where $H^{(a, b)}$ denotes the $p \times p$ matrix with $H^{(a, b)}_{ab} = 1$ and $H^{(a, b)}_{ij} = 0$ for $(i, j) \neq (a, b)$, and $e_i$ denotes the $i\supscr{th}$ canonical basis vector of $\sR^p$.  We must show that
    $$
      \{\incr{\Phi^\mcI_D}(B, \gamma)(H^{(a, b)}, 0)\}_{(a, b) \in E} \cup \{\incr{\Phi^\mcI_D}(B, \gamma)(0, e_i)\}_{1 \leq i \leq p}
    $$
    is a linearly independent set for all $(B, \gamma) \in \mathbf{B}(D) \times \sR^p_{>0}$.  Again, it is sufficient to consider the derivatives of the precision matrices $K^{(I)}$.
    
    We start with the directional derivative of $K^{(I)}$ in direction $(H^{(a, b)}, 0)$ for a pair $(a, b) \in E$.  This derivative is
    \begin{multline*}
      \incr{K^{(I)}}(B, \gamma)(H^{(a, b)}, 0) = \\ 
      - H^{(a, b)} R^{(I)} \diag(\gamma) R^{(I)} (\mathbbm{1} - B)^\transp - (\mathbbm{1} - B) R^{(I)} \diag(\gamma) R^{(I)} H^{(a, b)\transp} \ .
    \end{multline*}
    For a matrix $A \in \sR^{p \times p}$, the matrix $H^{(a, b)} A$ contains $A_{b \; \smallbullet}$ as the $a\supscr{th}$ row; all other rows are filled with zeros.  We then can see that
    $$
      \left[ R^{(I)} \diag(\gamma) R^{(I)} (\mathbbm{1} - B)^\transp \right]_{b \; \smallbullet} = \begin{cases} \gamma_b \left((\mathbbm{1} - B)_{\smallbullet \; b} \right)^\transp, & \textrm{if } b \notin I, \\ 0, & \textrm{otherwise.} \end{cases}
    $$
    Based on these considerations and the fact that $B$ is a strictly lower triangular matrix, one can then show that, if $b \notin I$, $\incr{K^{(I)}}(B, \gamma)(H^{(a, b)}, 0) = F^{(a, b)}$, where
    $$
      F^{(a, b)} := \gamma_b
      \left(
      \begin{array}{c@{}|@{}c@{}|@{}c}
        \text{\huge 0} &
        \begin{array}{c}
          0 \\ \vdots \\ 0 \\ -1 \\ B_{b+1,b} \\ \vdots \\ B_{a-1,b}
        \end{array}
        & \text{\huge 0} \\\hline
        \begin{array}{ccccccc}
          0 & \cdots & 0 & -1 & B_{b+1,b} & \cdots & B_{a-1,b}
        \end{array}
        & 2B_{ab} &
        \begin{array}{ccc}
          B_{a+1,b} & \cdots & B_{pb}
        \end{array}
        \\\hline
        \text{\huge 0} &
        \begin{array}{c}
          B_{a+1,b} \\ \vdots \\ B_{pb}
        \end{array}
        & \text{\huge 0}
      \end{array}
      \right)
    $$
    
    We continue with the calculation of the directional derivative of $K^{(I)}$ in direction $(0, e_b)$, $1 \leq b \leq p$.  In this less tedious case, we that
    $$
      \incr{K^{(I)}}(B, \gamma)(0, e_b) = \begin{cases} (\mathbbm{1} - B)_{\smallbullet \; b} \left( (\mathbbm{1} - B)_{\smallbullet \; b} \right)^\transp, & \text{if } b \notin I,\\ 0, & \text{otherwise.} \end{cases}
    $$
    This means that, for $b \notin I$, we have $\incr{K^{(I)}}(B, \gamma)(0, e_b) = G^{(b)}$, where
    $$
      G^{(b)} :=
      \left(
      \begin{array}{c|c}
        \text{\huge 0} & \text{\huge 0} \\\hline
        \text{\huge 0} & 
        \begin{array}{cccc}
          1 & -B_{b+1,b} & \cdots & -B_{pb} \\
          -B_{b+1,b} & & & \\
          \vdots & & \text{\huge *} & \\
          -B_{pb} & & &
        \end{array}
      \end{array}
      \right)
    $$
    It can easily be seen that the matrices $\{F^{(a, b)}\}_{a > b} \cup \{G^{(b)}\}_{1 \leq b \leq p}$ are linearly independent.  Since for each $b \in [p]$, there is some $I \in \mcI$ with $b \notin I$, we can finally conclude that the set
    $$
      \{\incr{\Phi^\mcI_D}(B, \gamma)(H^{(a, b)}, 0)\}_{(a, b) \in E} \cup \{\incr{\Phi^\mcI_D}(B, \gamma)(0, e_i)\}_{1 \leq i \leq p}
    $$
    is linearly independent, which proves the claim.
  \end{subprop}
\end{proof}

We have now shown that the parameter sets $M^\mcI_D$ are smooth embedded manifolds.  To be able to apply Theorem \ref{thm:bic-consistency}, it remains to show that two different parameter manifolds are not arbitrarily close.

\begin{proposition}
  \label{prop:models-separated}
  Let \mcI{} be a conservative family of targets, and let $D_1$ and $D_2$ be two DAGs that are \emph{not} \mcI-equivalent.  Assume that $\theta \in \mathcal{S} \times \mathcal{V} \times \mathcal{W}$ is a parameter vector with $\theta \in \overline{M^\mcI_{D_1}}$ and $\theta \in M^\mcI_{D_2}$.  Then also $\theta \in M^\mcI_{D_1}$ holds.
\end{proposition}
\begin{proof}
  each $j$, a unique parameterization $(B^{(j)}, \gamma^{(j)}) \in \mathbf{B}(D) \times \sR^p_{>0}$ such that $\theta^{(j)} = \Phi^\mcI_{D_1}(B^{(j)}, \gamma^{(j)})$.  The sequence $\left(B^{(j)}, \gamma^{(j)}\right)_{j \geq 1}$ must be bounded, otherwise the sequence $\theta^{(j)} = \Phi^\mcI_{D_1}(B^{(j)}, \gamma^{(j)})$ could not be bounded since $K^{(I)}$, $I \in \mcI$, are polynomials in $B$ and $\gamma$ (Definition \ref{def:exponential-parameterization}).  By the theorem of Bolzano-Weierstrass we therefore have a subsequence $\left(B^{(j_k)}, \gamma^{(j_k)}\right)$ that converges to some $(B, \gamma) \in \mathbf{B}(D) \times \sR^p_{\geq 0} = \overline{\mathbf{B}(D) \times \sR^p_{>0}}$.
  
  The parameterization $\Phi^\mcI_{D_1}$ has a continuous continuation on $\mathbf{B}(D) \times \sR^p_{\geq 0}$.  Therefore we have
  $$
    \theta^{(j_k)} = \Phi^\mcI_{D_1}(B^{(j_k)}, \gamma^{(j_k)}) \stackrel{k \to \infty}{\longrightarrow} \Phi^\mcI_{D_1}(B, \gamma) \ ,
  $$
  and $\Phi^\mcI_{D_1}(B, \gamma) = \theta$ holds because of the uniqueness of limits.
  
  It remains to show that $(B, \gamma) \in \mathbf{B}(D) \times \sR^p_{>0}$, that is, to show that $\gamma_b \neq 0$ for all $b \in [p]$.  Since \mcI{} is conservative, there is, for each $b \in [p]$, some $I \in \mcI$ such that $b \notin I$.  From Lemma \ref{lem:natural-parameters}, we know that
  $$
    \det K^{(I)}(B, \gamma) = \det \tilde{K}^{(I)} \prod_{a \notin I} \gamma_a \ ;
  $$
  since the prerequisite $\theta \in M^\mcI_{D_2}$ implies $\det K^{(I)} \ne 0$, we conclude that $\gamma_a \ne 0$ for all $a \notin I$.  This in particular implies $\gamma_b \ne 0$, which completes the proof.
\end{proof}

We have now shown that the parameter sets $M_D^\mcI$ of all DAGs $D$ fulfill the prerequisites of Theorem \ref{thm:bic-consistency}; an immediate consequence is the following corollary:

\begin{corollary}
  Consider model (\ref{eqn:basic-model2}) with the family of intervention targets \mcI.  Assume (A3) from Theorem \ref{thm:consistency}, and the estimator
  $$
    \hat{D} = \argmin_D - \ell_D(\hat{B}(D), \{\hat{\sigma}_k^2(D)\}_k; \data) + \frac{1}{2} \log(n) \dim(D) \ .
  $$
  Then: as $n \to \infty$,
  $$
    \PP[\hat{D}\ \mbox{is a minimum independence map}] \to 1\ ,
  $$
  where $\PP$ refers to the probability distribution under the true model.
\end{corollary}

As we noted in Section \ref{sec:structure-learning-bic}, every minimum independence map is \mcI-Markov equivalent to the true model if the true observational and all corresponding interventional densities are faithful.  In this case (that is, under the assumptions (A1) and (A2) of Section \ref{sec:structure-learning-bic}), the optimization problem in (\ref{eqn:MLEEQ}) almost surely has a \emph{unique} solution in the limit $n \to \infty$, namely the \mcI-Markov equivalence class of the true model (Theorem \ref{thm:consistency}).

\bibliographystyle{chicago}

\bibliography{references}

\begin{thebibliography}{}

\bibitem[\protect\citeauthoryear{Andersson, Madigan, and Perlman}{Andersson
  et~al.}{1997}]{Andersson1997Characterization}
Andersson, S., D.~Madigan, and M.~Perlman (1997).
\newblock {A} characterization of {M}arkov equivalence classes for acyclic
  digraphs.
\newblock {\em The Annals of Statistics\/} {\em 25}, 505--541.

\bibitem[\protect\citeauthoryear{Banerjee, {El Ghaoui}, and
  {d'Aspremont}}{Banerjee et~al.}{2008}]{Banerjee2008Model}
Banerjee, O., L.~{El Ghaoui}, and A.~{d'Aspremont} (2008).
\newblock {M}odel selection through sparse maximum likelihood estimation for
  multivariate {G}aussian or binary data.
\newblock {\em Journal of Machine Learning Research\/} {\em 9}, 485--516.

\bibitem[\protect\citeauthoryear{Chickering}{Chickering}{2002a}]{Chickering2002Learning}
Chickering, D. (2002a).
\newblock {L}earning equivalence classes of {B}ayesian-network structures.
\newblock {\em Journal of Machine Learning Research\/} {\em 3}, 445--498.

\bibitem[\protect\citeauthoryear{Chickering}{Chickering}{2002b}]{Chickering2002Optimal}
Chickering, D. (2002b).
\newblock {O}ptimal structure identification with greedy search.
\newblock {\em Journal of Machine Learning Research\/} {\em 3}, 507--554.

\bibitem[\protect\citeauthoryear{Chickering}{Chickering}{1996}]{Chickering1996Learning}
Chickering, D.~M. (1996).
\newblock {L}earning {B}ayesian networks is {NP}-complete.
\newblock In D.~Fisher and H.~Lenz (Eds.), {\em {L}earning from {D}ata:
  {A}rtificial {I}ntelligence and {S}tatistics {V}}, pp.\  121--130. Springer.

\bibitem[\protect\citeauthoryear{Cooper and Yoo}{Cooper and
  Yoo}{1999}]{Cooper1999Causal}
Cooper, G.~F. and C.~Yoo (1999).
\newblock {C}ausal discovery from a mixture of experimental and observational
  data.
\newblock In {\em {P}roc. {F}ifthteenth {C}onference on {U}ncertainty in
  {A}rtificial {I}ntelligence ({UAI} 1999)}, pp.\  116--125.

\bibitem[\protect\citeauthoryear{Eaton and Murphy}{Eaton and
  Murphy}{2007}]{Eaton2007Exact}
Eaton, D. and K.~Murphy (2007).
\newblock {E}xact {B}ayesian structure learning from uncertain interventions.
\newblock In {\em {P}roceedings of the {E}leventh {I}nternational {C}onference
  on {A}rtificial {I}ntelligence and {S}tatistics}, Volume~2, pp.\  107--114.

\bibitem[\protect\citeauthoryear{Friedman, Hastie, and Tibshirani}{Friedman
  et~al.}{2007}]{Friedman2007Sparse}
Friedman, J., T.~Hastie, and R.~Tibshirani (2007).
\newblock {S}parse inverse covariance estimation with the graphical {Lasso}.
\newblock {\em Biostatistics\/} {\em 9}, 432--441.

\bibitem[\protect\citeauthoryear{Haughton}{Haughton}{1988}]{Haughton1988Choice}
Haughton, D. D.~M. (1988).
\newblock {O}n the choice of a model to fit data from an exponential family.
\newblock {\em The Annals of Statistics\/} {\em 16\/}(1), 342--355.

\bibitem[\protect\citeauthoryear{Hauser and Bühlmann}{Hauser and
  Bühlmann}{2012}]{Hauser2012Characterization}
Hauser, A. and P.~Bühlmann (2012).
\newblock {C}haracterization and greedy learning of interventional {M}arkov
  equivalence classes of directed acyclic graphs.
\newblock {\em Journal of Machine Learning Research\/} {\em 13}, 2409--2464.

\bibitem[\protect\citeauthoryear{He and Geng}{He and Geng}{2008}]{He2008Active}
He, Y.-B. and Z.~Geng (2008).
\newblock {A}ctive learning of causal networks with intervention experiments
  and optimal designs.
\newblock {\em Journal of Machine Learning Research\/} {\em 9}, 2523--2547.

\bibitem[\protect\citeauthoryear{Hoyer, Janzing, Mooij, Peters, and
  Sch\"olkopf}{Hoyer et~al.}{2009}]{Hoyer2009Nonlinear}
Hoyer, P., D.~Janzing, J.~Mooij, J.~Peters, and B.~Sch\"olkopf (2009).
\newblock {N}onlinear causal discovery with additive noise models.
\newblock In {\em {A}dvances in {N}eural {I}nformation {P}rocessing {S}ystems
  21, 22nd {A}nnual {C}onference on {N}eural {I}nformation {P}rocessing
  {S}ystems ({NIPS} 2008)}, pp.\  689--696.

\bibitem[\protect\citeauthoryear{Kalisch and B{\"u}hlmann}{Kalisch and
  B{\"u}hlmann}{2007}]{Kalisch2007Estimating}
Kalisch, M. and P.~B{\"u}hlmann (2007).
\newblock {E}stimating high-dimensional directed acyclic graphs with the
  {PC}-algorithm.
\newblock {\em Journal of Machine Learning Research\/} {\em 8}, 613--636.

\bibitem[\protect\citeauthoryear{Kalisch, M\"achler, Colombo, Maathuis, and
  B\"uhlmann}{Kalisch et~al.}{2012}]{Kalisch2010Causal}
Kalisch, M., M.~M\"achler, D.~Colombo, M.~Maathuis, and P.~B\"uhlmann (2012).
\newblock {C}ausal inference using graphical models with the {R} package pcalg.
\newblock {\em Journal of Statistical Software\/} {\em 47 (11)}, 1--26.

\bibitem[\protect\citeauthoryear{Lauritzen}{Lauritzen}{1996}]{Lauritzen1996Graphical}
Lauritzen, S. (1996).
\newblock {\em {G}raphical {M}odels}.
\newblock Oxford University Press.

\bibitem[\protect\citeauthoryear{Maathuis, Kalisch, and B\"uhlmann}{Maathuis
  et~al.}{2009}]{Maathuis2009Estimating}
Maathuis, M., M.~Kalisch, and P.~B\"uhlmann (2009).
\newblock {E}stimating high-dimensional intervention effects from observational
  data.
\newblock {\em The Annals of Statistics\/} {\em 37}, 3133--3164.

\bibitem[\protect\citeauthoryear{Meinshausen and B\"uhlmann}{Meinshausen and
  B\"uhlmann}{2010}]{Meinshausen2010Stability}
Meinshausen, N. and P.~B\"uhlmann (2010).
\newblock {S}tability selection (with discussion).
\newblock {\em Journal of the Royal Statistical Society Series B\/} {\em 72},
  417--473.

\bibitem[\protect\citeauthoryear{Pearl}{Pearl}{1988}]{Pearl1988Probabilistic}
Pearl, J. (1988).
\newblock {\em {P}robabilistic reasoning in intelligent systems: networks of
  plausible inference}.
\newblock Morgan Kaufmann.

\bibitem[\protect\citeauthoryear{Pearl}{Pearl}{2000}]{Pearl2000Causality}
Pearl, J. (2000).
\newblock {\em {C}ausality: {M}odels, {R}easoning and {I}nference}.
\newblock Cambridge University Press.

\bibitem[\protect\citeauthoryear{Peters, Mooij, Janzing, and
  Sch\"olkopf}{Peters et~al.}{2011}]{Peters2011Identifiability}
Peters, J., J.~Mooij, D.~Janzing, and B.~Sch\"olkopf (2011).
\newblock {I}dentifiability of causal graphs using functional models.
\newblock In {\em {P}roceedings of the 27th {C}onference on {U}ncertainty in
  {A}rtificial {I}ntelligence ({UAI} 2011)}.

\bibitem[\protect\citeauthoryear{Sachs, Perez, Pe'er, Lauffenburger, and
  Nolan}{Sachs et~al.}{2005}]{Sachs2005Causal}
Sachs, K., O.~Perez, D.~Pe'er, D.~Lauffenburger, and G.~Nolan (2005).
\newblock {C}ausal protein-signaling networks derived from multiparameter
  single-cell data.
\newblock {\em Science\/} {\em 308\/}(5721), 523--529.

\bibitem[\protect\citeauthoryear{Schwarz and Köckler}{Schwarz and
  Köckler}{2006}]{Schwarz2006Numerische}
Schwarz, H.~R. and N.~Köckler (2006).
\newblock {\em {N}umerische {M}athematik\/} (6th ed.).
\newblock Stuttgart: Vieweg + Teubner.

\bibitem[\protect\citeauthoryear{Shimizu, Hoyer, Hyv\"arinen, and
  Kerminen}{Shimizu et~al.}{2006}]{Shimizu2006Linear}
Shimizu, S., P.~Hoyer, A.~Hyv\"arinen, and A.~Kerminen (2006).
\newblock {A} linear non-{G}aussian acyclic model for causal discovery.
\newblock {\em Journal of Machine Learning Research\/} {\em 7}, 2003--2030.

\bibitem[\protect\citeauthoryear{Silander and Myllym\"aki}{Silander and
  Myllym\"aki}{2006}]{Silander2006Simple}
Silander, T. and P.~Myllym\"aki (2006).
\newblock {A} simple approach for finding the globally optimal {B}ayesian
  network structure.
\newblock In {\em {P}roceedings of the 22th {C}onference on {U}ncertainty in
  {A}rtificial {I}ntelligence ({UAI} 2006)}.

\bibitem[\protect\citeauthoryear{Spirtes, Glymour, and Scheines}{Spirtes
  et~al.}{2000}]{Spirtes2000Causation}
Spirtes, P., C.~Glymour, and R.~Scheines (2000).
\newblock {\em {C}ausation, {P}rediction, and {S}earch\/} (Second ed.).
\newblock MIT Press.

\bibitem[\protect\citeauthoryear{Tsamardinos, Brown, and Aliferis}{Tsamardinos
  et~al.}{2006}]{Tsamardinos2006Maxmin}
Tsamardinos, I., L.~E. Brown, and C.~F. Aliferis (2006).
\newblock {T}he max-min hill-climbing {B}ayesian network structure learning
  algorithm.
\newblock {\em Machine Learning\/} {\em 65\/}(1), 31--78.

\bibitem[\protect\citeauthoryear{Verma and Pearl}{Verma and
  Pearl}{1990}]{Verma1990Equivalence}
Verma, T. and J.~Pearl (1990).
\newblock {O}n the equivalence of causal models.
\newblock In {\em {P}roceedings of the 11th {C}onference on {U}ncertainty in
  {A}rtificial {I}ntelligence ({UAI} 1990)}, pp.\  220--227.

\end{thebibliography}

\end{document}